\documentclass[leqno]{compositio}

\usepackage{amssymb,mathrsfs}
\usepackage{amsmath,mathscinet}
\usepackage[normalem]{ulem}
\usepackage{amssymb,amstext,titletoc,fancyhdr,color,xcolor,calc,graphicx}

\theoremstyle{plain}
\newtheorem{theorem}{Theorem}[section]
\newtheorem{lemma}[theorem]{Lemma}
\newtheorem{proposition}[theorem]{Proposition}

\newtheorem{corollary}[theorem]{Corollary}

\theoremstyle{definition}
\newtheorem{definition}[theorem]{Definition}
\newtheorem{example}[theorem]{Example}

\theoremstyle{remark}
\newtheorem{remark}[theorem]{Remark}

\numberwithin{equation}{section}

\def\N{{\mathbb N}}

\def\C{{\mathbb C}}

\def\Z{{\mathbb Z}}

\newcommand{\beqnn}{\begin{equation}}
\newcommand{\eeqnn}{\end{equation}}
\newcommand{\eb}{\begin{enumerate}}
\newcommand{\ee}{\end{enumerate}}
\newcommand{\bbm}{\begin{bmatrix}}
\newcommand{\ebm}{\end{bmatrix}}
\newcommand{\bpm}{\begin{pmatrix}}
\newcommand{\epm}{\end{pmatrix}}

\newcommand{\bi}{\begin{itemize}}
\newcommand{\ei}{\end{itemize}}
\newcommand{\beq}{\begin{eqnarray*}}
\newcommand{\eeq}{\end{eqnarray*}}

\newcommand{\beqq}{\begin{eqnarray}}
\newcommand{\eeqq}{\end{eqnarray}}
\newcommand{\beqn}{\begin{eqnarray}}
\newcommand{\eeqn}{\end{eqnarray}}

\DeclareMathOperator{\Tr}{Tr}

\begin{document}

\title[Resolving singularities \& monodromy reduction of  Fuchsian connections]{Resolving singularities and monodromy reduction of  Fuchsian connections${}^\ast$}

\author{Yik-Man Chiang}
\email{machiang@ust.hk}
\address{Department of Mathematics, The Hong Kong University of Science and Technology,
Clear Water Bay, Kowloon, Hong Kong SAR}
\author{Avery Ching}
\email{maaching@ust.hk}
\address{Department of Mathematics, The Hong Kong University of Science and Technology, Clear Water Bay, Kowloon, Hong Kong SAR}
\author{Chiu-Yin Tsang}
\email{h0347529@connect.hku.hk}
\address{Department of Mathematics, The University of Hong Kong, Pokfulam Road, Hong Kong SAR}

\dedication{Dedicated to the memory of Richard A.  Askey}


\begin{abstract} {We study monodromy reduction of Fuchsian connections from a sheave theoretic viewpoint, focusing on the case when a singularity of a special connection with four singularities has been resolved. The main tool of study is {based on} a bundle modification technique due to Drinfeld and Oblezin. This approach via invariant spaces and eigenvalue problems allows us not only to explain Erd\'elyi's classical infinite hypergeometric expansions of solutions to Heun equations, but also to obtain new expansions not found in his papers.  As a consequence,  a geometric proof of Takemura's eigenvalues inclusion theorem is obtained.  Finally, we observe a precise matching between the monodromy reduction criteria giving those special solutions of Heun equations and that giving classical solutions of the Painlev\'e VI equation.}
\end{abstract}

\maketitle

 \vspace*{6pt}\tableofcontents

\section{Introduction}
It has  been observed recently that the use of closed form solutions to Fuchsian equations with an apparent singularity led to  significant simplification in solving certain \textit{free moving boundary problems} \cite{Craster1996,Craster1997,CrasterHoang1998, SC}. In fact, resolving apparent singularities of Fuchsian equations is also  \cite{STW} related to automorphic forms and geometric properties of special functions   \cite{BM2015,CKLZ, Kimura, Maier} (see also \cite{Hautot3}). Indeed,  Erd\'elyi's study \cite{Erdelyi1, Erdelyi2} of  global monodromy groups of Heun equations by finding ``global hypergeometric expansion solutions"  was also along this ``removing apparent singularities paradigm".  However, in order to overcome the ambiguity of the interplay between the local and global aspects of solutions of Fuchsian equations typically 
using classical language, it is the purpose of this article to apply sheaf theoretic language to study geometric aspects of
 Fuchsian connections where one of their singularities becomes apparent (i.e., being resolved).  As consequences of our study, we recover Takemura's \textit{eigenvalue inclusion theorem} \cite{Take4}  with a much simplified geometric proof,  and moreover,  new hypergeometric expansions of solutions to Heun equations other than those already obtained by Erd\'elyi  are derived.

We show that ``resolving singularities'' for a given Fuchsian equation with four singularities, i.e., Heun equation, carries a deeper meaning than what the formulae could indicate. This is best described as a morphism between bundles equipped with appropriately defined connections that have one less singularity, i.e., a singularity of the original connection is being resolved by a singular gauge transformation. This phenomenon of removing singularity
is inline with the spirit of resolving singularities of algebraic curves. In particular, our sheaf theoretic approach does not only allow us to give an explanation of the origin of the various local hypergeometric function expansions of the Heun equation given by Erd\'elyi  \cite{Erdelyi1, Erdelyi2}, but it also allows us to understand the mechanism of Erd\'elyi's expansion thus
enabling us to derive a new expansion, thus giving a unified  theory of the classical works by Erd\'elyi in the geometric language of Heun connections. For instance, in \cite[p. 51]{Erdelyi1} (and a similar one in \cite[p. 63]{Erdelyi2}), Erd\'elyi proposed a solution {scheme} of a certain Heun equation in the form
	\begin{equation}\label{E:sum_1942}
		y=\sum_n c_n\, P
			\left(
			\begin{matrix}
			0                     & 1        & \infty\\
			n                     & 0        & \alpha \\
			\delta-\alpha-\beta-n & 1-\delta & \beta
			\end{matrix}
			;\ x
			\right)
	\end{equation}
in which the Riemann scheme $P$ denotes  functions satisfying a suitable hypergeometric equation. In other words, for each small open set $U$, denoted by $\mathcal{A}(U)$ and $\mathcal{B}(U)$ the spaces of solutions of that hypergeometric and Heun equations respectively. Erd\'elyi's expansion is in fact a collection of maps $\mathcal{A}(U)\to\mathcal{B}(U)$ which are compatible with analytic continuation to a neighbouring open set (see e.g.,  \cite[\S3.1-3.2]{SC}). In modern language, it is a morphism of sheaves, and in particular it is a morphism of local systems in this context. It is immediate that morphisms of local systems come from morphisms of flat bundles. One useful type of morphisms was exhibited by Oblezin \cite{Oblezin,Oblezin1} in order to 
 show that the classical contiguous relations of hypergeometric functions can be explained by a technique of bundle modification originated from Drinfeld \cite{Drinfeld1}. 
Simply speaking, when one of the singularities  becomes apparent,
we interpret this modification as an inclusion of spaces of certain kind of sections into itself with a faster growth rate.

 We have {achieved this by} singling out {the monodromy reduction (\textbf{WAS}$(m)$) (as well as two other conditions (\textbf{WGRM}) and (\textbf{LR}$(n)$). See the description below.)} of the Heun connections and showed that their different combinations with the language of morphisms of suitably defined flat bundles yields a geometric interpretation of Erdelyi's hypergeometric expansions, including the new ones not found in \cite{Erdelyi1, Erdelyi2}. 
 These morphisms relate the corresponding eigensections and eigenvalues of certain connections with different signatures enumerated by $m\in\mathbb{N}$ (Theorem \ref{T:geo_takemura}). An immediate observation from this study is a geometric formulation and proof of Takemura's remarkable \textit{eigenvalues inclusion theorem} \cite[Theorem 5.3]{Take4}  that we have already mentioned in the first paragraph.  Moreover, such a language of flat bundles identifies the \textit{new} common conditions of resolving singularity and the degeneration of the monodromy group of the Heun equation to the reduction of the sixth Painlev\'e equation $\mathbf{P_{VI}}$.

We start by considering the \textit{Heun type-connection}
	\begin{equation}\label{E:Heun_connection}
		\nabla=d-\Big[\dfrac{A_0}{x}+\dfrac{A_1}{x-1}+\dfrac{A_a}{x-a}\Big]\, dx,
	\end{equation}
where each residue matrix $A_k\ (k=0,\, 1,\, a,\, \infty)$ is diagonalisable. Its eigensections can be written as a finite sum of the form
	\begin{equation}\label{E:formal}
		\sum_{n}c_n \, (\nabla_v)^n\, A_a\,s,
	\end{equation}
and the exact forms of these sums are determined by the eigenvalue differences of the residue matrices $A_k\ (k=0,\, 1,\, a,\, \infty)$ and the eigenvalues of the $\nabla_v$. Without loss of generality,  we shall normalise the matrices so that the eigenvalues of the $A_0,\, A_1,\, A_a$ are $\{0,\, 1-\gamma\}, \{0,\, 1-\delta\}, \{0,\, 1-\epsilon\}$ respectively.
The above theory to be developed in this article  gives an explanation of Erd\'elyi's local solutions \cite[1942]{Erdelyi1}, \cite[1944]{Erdelyi2} to the Heun equation 
\begin{equation}\label{E:heun}
        D\, y:= x(x-1)(x-a) \frac{d^2y}{dx^2}+
         x(x-1)(x-a)\Big(\frac{\gamma}{x}+\frac{\delta}{x-1}+\frac{\epsilon}{x-a}\Big)
         \frac{dy}{dx}+
         \alpha\beta xy=qy,
\end{equation}
 where the parameters satisfy the Fuchsian constraint $\alpha+\beta-\gamma-\delta-\epsilon+1=0$. For example, Erd\'elyi found an infinite sum of hypergeometric functions
 	\begin{equation}\label{E:erdelyi}
		\sum_{n=0}^\infty \mathcal{C}_n(q)\,
x^n
\sideset{_2}{_1}{\operatorname{F}}\left({\begin{matrix}
                 \alpha+n,\ \beta+n\\
                  \alpha+\beta-\delta+2n+1
                 \end{matrix}};x\right),
		\end{equation}
which represents a local solution expanded about the singularity $x=0$ with the exponent $0$,  where the coefficients $\mathcal{C}_n$ satisfy a certain three-term recursion \cite[pp. 54--55]{Erdelyi1} (also see Appendix~\ref{A:erdelyi}).

We note that the summations \eqref{E:sum_1942} and \eqref{E:erdelyi} can be regarded as classical forms of \eqref{E:formal}.
It is known that when the \textit{accessory parameter} $q$ in  \eqref{E:heun} is \textit{suitably} chosen, the  infinite sum \eqref{E:erdelyi} converges in a certain domain $\Omega_1^-$ 
containing two singularities and a third singularity situates on its circumference (see Appendix~\ref{A:erdelyi} for more details). Erd\'elyi hinted that such hypergeometic expansions carry more monodromy data than the usual power series expansions do and that would pave a way to study the monodromy group of the Heun equation. Indeed, when the accessory $q$ in \eqref{E:heun} is so chosen, then the original expansion \eqref{E:erdelyi} which represents an analytic solution around the origin is now being analytic continued to a larger region containing the points $x=1$ and $x=\infty$. This is equivalent to the  \textit{simultaneous diagonalisation} of the monodromies of \eqref{E:heun} at $x=1$ and $x=\infty$, or more generally at
any two out of the four singularities. We call this kind of monodromy reduction \textit{weak global condition of reducibility of monodromy} (\textbf{WGRM}).

When the monodromy of the Heun-type connection \eqref{E:Heun_connection}
meets the  \textit{weak apparent singularity condition} (\textbf{WAS}$(m)$), namely when the matrix $A_a$ possesses $0$ and a non-zero integer as eigenvalues, then its eigen-sections possess terminated expansions \eqref{E:formal}.
This follows from Theorem \ref{T:resolving_sing} below that if the connection \eqref{E:Heun_connection} meets the (\textbf{WAS}$(m)$), then there is a sheaf of locally free $\mathcal{O}_{\mathbb{P}^1\backslash\{0,1,\infty\}}-$module $\mathcal{E}$, such that
	\[
			\nabla: \mathcal{E}\to \mathcal{E}\otimes \Omega_{\mathbb{P}^1\backslash\{0,1,\infty\}}
	\]
	has the prescribed connection matrix as well as the matrix
	\[
		-\Big[\dfrac{A_0+A_{a}}{x}+\dfrac{A_1}{x-1}\Big]\, dx
	\]
relative to two different frames.
In particular, the fourth singularity $x=a$ is \textit{removed}. Its eigen-sections are
 terminated hypergeometric expansions \eqref{E:erdelyi} given by Erd\'elyi. 
A similar argument gives another expansion
	\begin{equation}\label{E:erdelyi_2}
		\sum_{n=0}^\infty \mathcal{C}_n\,
\sideset{_2}{_1}{\operatorname{F}}\left({\begin{matrix}
                 \lambda+n,\ \mu-n\\
                  \gamma
                 \end{matrix}};x\right),
		\end{equation}
where $\lambda+\mu=\gamma+\delta-1$ and the coefficients $\mathcal{C}_n$ satisfy a certain three-term recursion \cite[p. 64]{Erdelyi2}. This can be derived from the variation that the matrix of $\nabla$ relative to another frame is
	\[
		-\Big[\dfrac{A_0}{x}+\dfrac{A_1}{x-1}\Big]\, dx.
	\]
We derive a number of infinite expansions not found in the works of Erd\'elyi such as 
	\[
			\sum_{n=0}^\infty \mathcal{C}_n\,
			(x-1)^{n}
\sideset{_2}{_1}{\operatorname{F}}\left({\begin{matrix}
                 \alpha+ n,\ \beta+ n\\
                  \gamma
                 \end{matrix}};\, x\right)
	\]
by having
	\[
		-\Big[\dfrac{A_0}{x}+\dfrac{A_1+A_{a}}{x-1}\Big]\, dx
	\]
as the matrix of $\nabla$ relative to yet another frame.

Notice that the Heun equation \eqref{E:heun} can be derived from the Heun-type connection $\nabla$ in \eqref{E:Heun_connection}, as a member of an isomonodromic family of connections where the necessary appearance of the fifth \textit{apparent singularity} is located at $\infty$.

We list here the three kinds of \textit{degenerations} mentioned:
		\begin{itemize}
		\item \textit{weak global condition of reducibility of monodromy}  (\textbf{WGRM};  \textit{Definition} \ref{D:WGRM}).
		\item \textit{weak apparent singularity condition} (\textbf{WAS}$(m)$;  \textit{Definition} \ref{D:WAS}).
		\item \textit{local condition for reducibility of monodromy\ } (\textbf{LR}$(m))$; \textit{Definition} \ref{D:LR}).
	\end{itemize}

One notices that when $\epsilon=-m$ (where $m\in \mathbb{N}$) (i.e., this corresponds to the degenerate condition ($\mathbf{WAS}(m)$), the
set of multi-valued functions with the same monodromy as a certain hypgerometric equation,
and  with a pole of order  at most $m+1$  at $a$ is invariant under the operator $D$ defined in \eqref{E:heun}. One of the main purposes of this article is to investigate the \textit{eigen-value problem}
	\begin{equation}\label{E:eigenvalue_prob_1}
		D\, y=q\, y
	\end{equation} by combining the study of monodromy and growth rate in a  sheave theoretic methodology, the eigenvalue problem above becomes the {eigenvalue} problem
	\begin{equation}\label{E:eigenvalue_prob_2}
		\nabla_v\, s=\lambda\, s.
	\end{equation}
Here $s$ is a section of a bundle constructed by tensoring a local system with a divisor line bundle, which yields the description of a function space of those with a specific monodromy and growth rate at singularities.

	In Theorem \ref{T:resolving_sing}, we construct the flat bundle $(\mathcal{E}, \nabla)$ in which the section $s$ lives. Then in  Theorem \ref{T:eigenvalues_a} we prove that the $m+1$ eigenvalues $\lambda$ above are precisely the numbers (or are equally spaced)  $0$, $1$, $\cdots$, $m$.  The idea of the proof rests on the construction of a sequence of injective bundle morphisms
	\[
		(\mathcal{E}_m, \nabla_m)\ \longrightarrow\ (\mathcal{E}_{m-1}, \nabla_{m-1}) \   \longrightarrow \cdots  \longrightarrow\ (\mathcal{E}, \nabla)
	\]
	by a technique of bundle modification due to Drinfeld \cite{Drinfeld1} and Oblezin \cite{Oblezin,Oblezin1}\footnote{Oblezin attributed  his bundle construction has its origin from Drinfeld \cite{Drinfeld1}.}. This sequence of bundle morphisms mimics the triangulation of a matrix. Moreover, each of these eigen-sections  
can be written as a formal sum
	\[
		\sum_n c_n^k\, (\nabla_n)^n A_as,\qquad k=0, 1,\cdots, m
	\]
	for some local section $s$ of a sheaf of horizontal sections of a hypergeometric type connection. This gives a geometric interpretation of the $m+1$ eigen-solutions of \eqref{E:heun} studied since \cite{Erdelyi1,Erdelyi2}. 
	
Moreover, the expansions of Erd\'elyi \eqref{E:sum_1942} suggest that some degenerated Heun equations are related to hypergeometric equations. 
Kimura \cite{Kimura1} (1970) 
makes this relation explicit by means of gauge transformations. In fact these explicitly constructed (singular) gauge transformations are morphisms between flat bundles coming from bundle modifications. These morphisms are isomorphisms away from a singular point, which suit the purpose of keeping the properties of a connection unchanged away from the singular point. We will see that it is the geometric context behind the (singular) gauge transformations studied by Kimura \cite{Kimura1}.	
Then we have discovered the startling result that the totality of these conditions are
\textit{identical} to those that lead to the degeneration of the Painlev\'e VI found by Okamoto \cite {Okamoto}.
In \cite{CCT2} 
the authors showed the conditions for which the degeneration of the Darboux (i.e., a periodic Heun) connection 
matches the conditions for the degeneration of Painlev\'e VI.

The article is organized as follows. We begin by recalling some useful results for flat vector bundles, monodromies of local systems and Kummer sheaves, and 
giving the construction of Drinfeld-Oblezin bundle modifications that suits our purpose in this article in Section \ref{S:sheaf}. Then we turn our attention to hypergeometric connections and Heun connections in Section \ref{S:hypergeo} and Section~\ref{S:heun} respectively. In particular, we introduce the first category of monodromy reduction (\textbf{WGRM}) in terms of simultaneous diagonalisation of monodromy matrices at any two singularities amongst all the singularities. This allows us to review the infinite continued fractions that Erd\'elyi used as a criteria for enlarging the domain of convergence of his classical infinite hypergeometric expansions to solutions of Heun equations. We shall introduce the second category of  monodromy degeneration (\textbf{WAS}$(m)$) for an integer $m\in\mathbb{N}$ for \eqref{E:Heun_connection} in Section \ref{S:eigenspaceI}. As our first main result in this section, this formulation enables us to ``resolve" the singularity $x=a$ of \eqref{E:Heun_connection} (Theorem \ref{T:resolving_sing}). We also obtain other criteria which are equivalent to  (\textbf{WAS}$(m)$) for resolving singularities of \eqref{E:Heun_connection} other than $x=a$ in this section. Moreover,  we demonstrate that our geometric criterion (\textbf{WAS}$(m)$) can lead to all the known hypergeometric function expansions by, including new expansions which are not found in,  Erd\'elyi \cite{Erdelyi1, Erdelyi2} as applications. In particular, Example 
\ref{E:bad_infinty} is instrumental in giving a geometric proof of Takemura's eigenvalues theorem  in Section~\ref{S:TwoType} later. This section also contains the next main result of this paper, namely that for a given $m\in\mathbb{N}$, the criterion (\textbf{WAS}$(m)$) implies that there exist $m+1$ local sections $s_k\ (k=0,\,\cdots, m)$ on any  given open set such that the eigenvalue problem \eqref{E:eigenvalue_prob_2} always admit equally spaced eigenvalues $\lambda_k \ (k=0,\,\cdots, m)$ (Theorem \ref{T:eigenvalues_a}). 
Section \ref{S:heunpolyn} introduces the third category of monodromy reduction (\textbf{LR}$(n)$) for $n\in \mathbb{N}$ to the Heun connection \eqref{E:Heun_connection}. The criterion (\textbf{LR}$(n)$)  together with (\textbf{WGRM}) allow us to describe globally defined solutions for the Heun connection \eqref{E:Heun_connection} which contain the Heun polynomials  as special cases. Section \ref{S:TwoType} studies the combined geometric effects of all three categories of monodromy reductions to the the Heun connection \eqref{E:Heun_connection} introduced, namely, the fulfilment of (\textbf{WGRM}), (\textbf{WAS}$(m)$) and (\textbf{LR}$(n)$). In particular, we offer a geometric proof to Takemura's eigenvalue inclusion theorem (Theorem \ref{EigenInc}) as a consequence. In Section~\ref{S:Pvi}, we 
observe that a complete matching between the criteria giving monodromy reduction  of Heun connection and  classical solutions of the Painlev\'e VI equation under the
two categories (\textbf{WAS}$(m)$) and (\textbf{LR}$(n)$).
The paper ends with some  remarks in Section~\ref{S:conclusion}.

\section{Flat bundles} \label{S:sheaf}

We give preliminaries of the theory of flat connections and local systems, which are useful for understanding the geometric context of some classical special functions.
For more details, the readers may refer to the  survey by Malgrange \cite[Chapter IV]{BGKHME}.

\subsection{Flat connections and local systems}

Throughout this article, we let $X$ be a complex manifold. Indeed it is a (non-compact) Riemann surface except in Section \ref{S:Pvi}, where the underlying space $X$ has dimensional two. The sheaf of analytic functions on $X$ is denoted by $\mathcal{O}_X$ (or just $\mathcal{O}$ when there is no ambiguity). For each $p\in X$, the stalk $\mathcal{O}_{X,p}$ is a local ring whose maximal ideal is denoted by $\mathfrak{m}_p$. The sheaf of holomorphic one-forms on $X$ is always denoted by $\Omega_X$.

\begin{definition}
Let $\mathcal{E}$ be a locally free sheaf of $\mathcal{O}_X-$module. A connection $\nabla$ in $\mathcal{E}$ is a $\mathbb{C}-$linear morphism of sheaves $\nabla:\mathcal{E}\to\Omega_X\otimes\mathcal{E}$ such that for each open $U\subset X$,
\[
\nabla(fs)=f\nabla(s)+df\otimes s\hspace{1cm}\mbox{for every }f\in\mathcal{O}_X(U)\mbox{ and }s\in\mathcal{E}(U).
\]
If $(\mathcal{E},\nabla)$, $(\mathcal{E}',\nabla')$ are locally free sheaves of $\mathcal{O}_X-$modules endowed with connections, a morphism (or (singular) gauge transformation) $\phi:(\mathcal{E},\nabla)\to(\mathcal{E}',\nabla')$ is an $\mathcal{O}_X-$linear morphism $\phi:\mathcal{E}\to\mathcal{E}'$ such that the diagram
\[
\begin{array}{rcl}
\mathcal{E}&\stackrel{\nabla}{\xrightarrow{\hspace*{.6cm}}}&\Omega_X\otimes\mathcal{E}\\
\phi\Big\downarrow& &\Big\downarrow\iota\otimes\phi\\
\mathcal{E}'&\stackrel{\nabla'}{\xrightarrow{\hspace*{.6cm}}}&\Omega_X\otimes\mathcal{E}'
\end{array}
\]
commutes.
\end{definition}

Given a connection $\nabla$ in a locally free sheaf of $\mathcal{O}_X-$module $\mathcal{E}$, it extends to $\nabla:\Omega_X^k\otimes\mathcal{E}\to\Omega_X^{k+1}\otimes\mathcal{E}$ by enforcing the generalized \textit{Leibniz rule}
\[
\nabla(\alpha s)=d\alpha\; s+(-1)^k\alpha\wedge\nabla(s)\hspace{1cm}\mbox{for each section }\alpha\in\Omega_X^k(U)\mbox{ and }s\in\mathcal{E}(U).
\]
Then the \textit{curvature} of $(\mathcal{E},\nabla)$ is defined by $\nabla\circ\nabla$. The connection $\nabla$ is \textit{flat}, or \textit{integrable}, if its curvature vanishes. This flatness condition is void if $X$ is a Riemann surface as there is no holomorphic two-form there. This notion of flat bundle is a coordinate-free description of classical ODEs in the complex domain, or holonomic system of PDEs in case $X$ has a higher dimension. However, description with a choice of coordinates is also desirable for our purpose in certain applications as described in the following paragraph.

Choose a small open set $U\subset X$ such that $\mathcal{E}(U)\cong\mathcal{O}_X(U)^r$. Then $\mathcal{E}(U)$ has an $\mathcal{O}_X(U)-$basis $s_1$,..., $s_r\in\mathcal{E}(U)$. Let $\omega_{ij}\in\Omega_X(U)$ be one-forms such that
\[
\nabla(s_j)=\sum_j\omega_{ij}s_j\hspace{1cm}\mbox{for all }i.
\]
Then the matrix $\{\omega_{ij}\}_{1\leq i,j\leq r}$ with entries in $\Omega_X(U)$ is called the \textit{connection matrix} of $\nabla$ relative to the frame $s_i$, which is an explicit description of $\nabla$.
\bigskip

\begin{definition}
A sheaf of $\mathbb{C}-$vector spaces $\mathcal{L}$ on $X$ is locally constant, or a \textit{local system}, if there exists an open covering $\{U_\alpha\}$ of $X$ such that $\mathcal{L}|_{U_\alpha}$ is a constant sheaf for all $\alpha$. A morphism between local systems on $X$ is simply a $\mathbb{C}-$linear morphism between sheaves.
\end{definition}
\bigskip

Given a flat bundle $(\mathcal{E},\nabla)$, the sheaf $\mathcal{L}$ defined by
\[
\mathcal{L}(U)=\{s\in\mathcal{E}(U):\nabla s=0\}\hspace{1cm}\mbox{for all open }U\subset X
\]
is a local system, which is known as the sheaf of horizontal sections of $(\mathcal{E},\nabla)$. When a classical ODE (or a holonomic system of PDEs) is formulated as a flat bundle, it has special solutions when this flat bundle (or its sheaf of horizontal sections) has a proper non-trivial sub-object. We say such a flat bundle (or the corresponding local system) is \textit{reducible}.
\bigskip

\subsection{Monodromy of a local system}

\begin{definition}
Let $\mathcal{L}$ be a local system on $X$. Given a path $\gamma:[0,1]\to X$, one can cover the image of $\gamma$ by open sets $U_1,\cdots,U_n$ such that $\mathcal{L}|_{U_i}$ is constant and $U_i\cap U_{i+1}$ is nonempty for each $i$. This induces an isomorphism between fibers (or stalks) $\mathcal{L}_{\gamma(0)}\to\mathcal{L}_{\gamma(1)}$. This isomorphism depends only on the homotopy class of $\gamma$ and thus induces a $\pi_1(X,x_0)-$module structure on $\mathcal{L}_{x_0}$, which is called the \textit{monodromy} of $\mathcal{L}$. In case $\mathcal{L}$ is the sheaf of horizontal sections of a flat bundle, we speak of the monodromy of such a flat bundle directly.
\end{definition}

It is elementary to see that the above construction can be extended to a morphism between two flat bundles  to yield a $\pi_1(X)-$linear map between their monodromies. Hence we can verify the following well-known statement.
\bigskip

\begin{lemma}
Monodromy defines a covariant functor from the category of flat bundles on $X$ to the category of $\pi_1(X)-$modules.
\end{lemma}

There is no guarantee that such a functor is an equivalence of categories, unless it is restricted to a class of flat bundles having ``regular singularities".

\begin{definition}
Let $\overline{X}$ be a smooth projective variety, $X$ is the complement of a normal crossing divisor $D$ in $\overline{X}$. Denote the sheaf of holomorphic one-forms in $X$ having at most log poles along $D$ by $\Omega_{\overline{X}}(\log D)$. Fix a locally free sheaf of $\mathcal{O}_X-$module $\mathcal{E}$. We say that the connection $\nabla:\mathcal{E}\to\Omega_X\otimes\mathcal{E}$ has a log singularity (or regular singularity) along $D$ if it can be lifted to \\ $\nabla:\mathcal{E}\to\Omega_{\overline{X}}(\log D)\otimes\mathcal{E}$. Moreover, a connection is \textit{Fuchsian} if it has at most regular singularities.
\end{definition}

There has been already a coordinate-free description of regular singularities in a classical language (see e.g., \cite{Poole}). That is, an ODE has a regular singular point at $p$ if all its solutions have at most polynomial growth at $p$. Now we see that such differential equations are important as they are completely characterized by their monodromies, up to gauge transformations. We quote the following deep result by Deligne  \cite[Corollary to Theorem 5.9]{deligne}, see also \cite{BGKHME}.

\begin{theorem}
Monodromy is an equivalence from the category of flat bundles on $X$ with at most regular singularities to the category of $\pi_1(X)-$modules.
\end{theorem}

Finally we note that the notion of ``residues" is central in the description of connections with log singularities.

\begin{definition}\label{D:residue}
Given a small open set $U\subset X$, $x=(x_1,\cdots,x_n):U\to\mathbb{C}^n$ is a local coordinate, and  $D$ is a divisor so that $D\cap U$ is cut out by $x_1$. Let $\mathcal{E}$ be a sheaf of locally free $\mathcal{O}_X-$module of rank $N$. Suppose that $\nabla:\mathcal{E}\to\Omega_{\overline{X}}(\log D)\otimes\mathcal{E}$ is a connection with regular singularity along $D$. Its matrix relative to a certain frame is
\[
A\dfrac{dx_1}{x_1}\;\;\;\mod\Omega_{\overline{X}}(U)\otimes\mathbb{C}^{N\times N}
\]
for some $N\times N$ matrix $A$. Such an $A$ does not depend on the choice of the local coordinate $x$ and is called \textit{the residue of $\nabla$ along $D$}. In this article, this residue is denoted by Res$_D\nabla$. 
In case $\nabla$ is a flat connection in the trivial bundle $\mathcal{O}^N_X$ over $X=\mathbb{P}^1\backslash\{a_1,\cdots,a_r,\infty\}$ so that Res$_{a_i}\nabla=-A_i$, the matrix of $\nabla$ relative to the standard frame is
\[
-\Big[\dfrac{A_1}{x-a_1}+\cdots+\dfrac{A_r}{x-a_r}\Big]dx.
\]
The classical Riemann scheme of a second order linear differential equation derived from the equation $\nabla f=0$ lists the eigenvalues of the matrices $A_1,\,\cdots, A_r$ as well as those of
		\begin{equation}\label{E:Fuchsian}
			A_{\infty}=-(A_1+\cdots+A_r).
		\end{equation}
	Moreover, the relation \eqref{E:fuchsian_comparison} replaces the classical \textit{Fuchsian relation}.
In this article, we always assume that the residues of all connections at every singular point are \textit{diagonalisable}.
\end{definition}

The well-known fact that the local monodromies of a flat connection with regular singularities can be read off from its residues is summarised in the following.
\bigskip

\begin{lemma}
Let $\overline{X}$ be a compact Riemann surface, $p\in\overline{X}$, $\mathcal{E}$ is a locally free sheaf of $\mathcal{O}_X-$module, $\nabla$ is a flat connection with regular singularity at $p$. Then the eigenvalues of the local monodromy of the loop around $p$ are the eigenvalues of $\exp[2\pi i\, \mathrm{Res}_p\nabla]$.
\end{lemma}

\subsection{Kummer sheaves} \label{kummer}

The simplest non-trivial special function includes $(x-a)^\mu$, where $\mu$ is not an integer (see e.g., \cite{yoshida}).  Indeed, it can be regarded as a section of a local system to be discussed in this subsection. As an example, the expression 
 	\[
		x^{1-c}{}_2F_1(1+a-c,\, 1+b-c,\, 2-c;\, x)
	\]
which contains the product of two sections $x^{1-c}$ and ${}_2F_1(1+a-c,\, 1+b-c,\, 2-c;\, x)
$, satisfies the standard hypergeometric equation.
\bigskip

\begin{definition}
For each point $a\in\mathbb{C}$, and $\mu\in\mathbb{C}$, the {\it Kummer sheaf } $\mathcal{K}_a^{\mu}$
is the sheaf of horizontal sections of $\big(\mathcal{O}_{\mathbb{P}^1\backslash\{a,\infty\}},d-\mu\dfrac{dx}{x-a}\big)$.
\end{definition}
\bigskip

\begin{lemma}
Suppose that $\mathcal{L}$ is a local system in a deleted neighbourhood of $0\in\mathbb{C}$, which is also the sheaf
of horizontal sections of a flat connection $\nabla$ with $\log$ singularity at $0$. For each $\mu\in\mathbb{C}$,
let $\mathcal{K}_0^{\mu}\otimes\mathcal{L}$ be the sheaf of horizontal sections of $\overline{\nabla}$. Then
\[
\mbox{Res}_0\overline{\nabla}=\mbox{Res}_0\nabla+\mu I.
\]
\end{lemma}
\bigskip

\begin{proof}
Let $s$ be a local section of $\mathcal{L}\otimes\mathcal{O}$. Then the germ of $\nabla(s)$ at $0$ satisfies
\[
\nabla(s)=\big[(\mbox{Res}_0\nabla)s\big]
	\otimes
\dfrac{dx}{x}\mod (\mathcal{L}\otimes\Omega)_0
\]
Thus, if $t$ is a local section of $\mathcal{K}_0^{\mu}$,
\[
	\begin{split}
	\overline{\nabla}(t\otimes s) &=[\mu t\otimes\dfrac{dx}{x}]\otimes s+t\otimes[(\mbox{Res}_0\nabla) s\otimes \dfrac{dx}{x}]\\
		&=t\otimes[(\mbox{Res}_0\nabla+\mu I)]s\otimes\dfrac{dx}{x}\mod (\mathcal{L}\otimes\mathcal{K}^{\mu}_{0}\otimes\Omega)_0.
	\end{split}
\]
\end{proof}
\bigskip

\begin{example}
In the classical theory of ordinary differential equations, the general Riemann equation with three regular singularities can be transformed to the
Gauss hypergeometric equation by a simple and well-known transformation of the dependent variable. This transformation is
done by tensoring with Kummer sheaves in our context as follows:

Let $\nabla$ be a flat connection in a rank-two bundle over $\mathbb{P}^1\backslash\{0,1,\infty\}$ having log
singularities at $0$, $1$  and $\infty$. We let $\lambda$, $\mu$ be one of the eigenvalues of Res$_0\nabla$
and Res$_1\nabla$ respectively. Suppose that $\mathcal{L}$ is the sheaf of horizontal sections of $\nabla$, and if
$\mathcal{K}_0^{-\lambda}\otimes\mathcal{K}_1^{-\mu}\otimes\mathcal{L}$ is the sheaf of horizontal sections
of $\overline{\nabla}$, then $0$ is an eigenvalue for both of Res$_0\overline{\nabla}$ and Res$_1\overline{\nabla}$.
\end{example}
\bigskip

\begin{example}
In the classical theory of second order Fuchsian differential equations (applicable to more general equations also),
the first order term (i.e., the term involving $y^\prime$) can be removed by a simple transformation of the dependent variable. We explain this operation by the geometric
context briefly as follows:

Let $\nabla$ be a flat connection in a rank-two bundle over $\mathbb{P}^1\backslash\{0,1,\infty\}$ having log
singularities at $0$, $1$ and $\infty$, so that $0$, $\lambda$ are the eigenvalues of Res$_0\nabla$ and
$0$, $\mu$ are the eigenvalues of Res$_1\nabla$. Suppose that $\mathcal{L}$ is the sheaf of horizontal sections
of $\nabla$. Then, the Wro\'nskian of $\mathcal{L}$ is
\[
\wedge^2\mathcal{L}\cong\mathcal{K}_0^{\lambda}\otimes\mathcal{K}_1^{\mu}.
\]
Denote $\mathcal{K}_0^{-\lambda/2}\otimes\mathcal{K}_1^{-\mu/2}$ by $(\wedge^2\mathcal{L})^{-1/2}$. Then
$\mathcal{L}\otimes(\wedge^2\mathcal{L})^{-1/2}$ is a rank-two local system with trivial Wro\'nskian.
\end{example}
\bigskip

\subsection{Drinfeld-Oblezin bundle modifications}

In this subsection, we revise the construction of two types of special bundles defined by Drinfeld  \cite{Drinfeld1} and later being applied to Heun equations \eqref{E:heun}  by Oblezin in \cite{Oblezin, Oblezin1}. It turns out that this idea is very useful for us to explain the geometry of those ``singular-gauge transformations'' that appeared in \cite{Kimura1} and in this article.

\bigskip

\begin{definition}\label{D:modification}
Let $X$ be a space (usually $\mathbb{P}^1$ for our purpose) and let $\mathcal{E}$ be a sheaf of $\mathcal{O}_X$-
module. Given the data
\[
\begin{array}{l}
p\in X\mbox{ and }\\
W\mbox{ is a }\mathcal{O}_p/\mathfrak{m}_p\mbox{-vector subspace of }\mathcal{E}_p,
\end{array}
\](we recall that $\mathfrak{m}_p$ denotes the maximal ideal of the  local ring of analytic functions $\mathcal{O}_p$ at $p$),
the \textrm{(}\textit{Drinfeld-Oblezin}) \textit{modification} of $\mathcal{E}$ at $(p,W)$ is the following subsheaf of $\mathcal{E}$:
\[
U\mapsto\{s\in \mathcal{E}(U):s_p\in W\}.
\]
\end{definition}
\bigskip

\begin{lemma}
Let $X$ be a Riemann surface. If $\mathcal{E}$ is a sheaf of locally free $\mathcal{O}_X$-module, then
the modification of $\mathcal{E}$ at any pair $(p,W)$ is also locally free.
\end{lemma}
\medskip

\begin{proof}
Let $\underline{\mathcal{E}}$ be the sheaf modification of $\mathcal{E}$ at $(p,W)$.
Given a free $\mathcal{O}_{X,p}$-module $\mathcal{E}_p$, it suffices to show that
$\underline{\mathcal{E}}_p$ is a free $\mathcal{O}_{X,p}$-module.
But $\mathcal{O}_{X,p}$ is a PID since $X$ is a Riemann surface
and hence $\underline{\mathcal{E}}_p$ is also free.
\end{proof}
\bigskip

\begin{definition}\label{D:modified_nabla}
Let $X$ be a Riemann surface, and let $\mathcal{E}$ be a sheaf of $\mathcal{O}_X$-module
equipped with a flat connection $\nabla$ with a log singularity at $p$, such that the residue of $\nabla$ at $p$ has two complementary invariant subspaces:
\[
\mbox{Ker(Res}_p\nabla-\lambda I)\quad\mbox{ and }\quad W,
\]
for some $\lambda\in\mathbb{C}$ (In particular, the condition holds when the residue of $\nabla$ at $p$ is diagonalisable.). We denote  the (lower) modification of $\mathcal{E}$ at $(p,W)$ (Definition \ref{D:modification})
 by $\underline{\mathcal{E}}$ and the connection $\underline{\nabla}$ is defined by
  \[
\underline{\nabla}:\underline{\mathcal{E}}\hookrightarrow\mathcal{E}\stackrel{\nabla}{\longrightarrow }\mathcal{E}\otimes\Omega_X
\longrightarrow \underline{\mathcal{E}}\otimes\Omega_X
\]
where the last arrow  is the projection $\mathcal{E}\to\underline{\mathcal{E}}$ coming from the invariant subspaces above.  That is, the flat connection $(\underline{\mathcal{E}},\underline{\nabla})$ fits into the commutative diagram
\[
\begin{array}{ccc}
\underline{\mathcal{E}}&\stackrel{\underline{\nabla}}{\xrightarrow{\hspace*{.6cm}}}&\underline{\mathcal{E}}\otimes\Omega_X\\
\Big\downarrow&&\Big\downarrow\\
\mathcal{E}&\stackrel{{\nabla}}{\xrightarrow{\hspace*{.6cm}}}&\mathcal{E}\otimes\Omega_X
\end{array}
\]
so that a morphism $(\underline{\mathcal{E}},\underline{\nabla})\to(\mathcal{E},\nabla)$ is obtained.
\end{definition}
\bigskip

The essential idea of the following theorem is derived from an example of Oblezin's construction \cite[p. 114]{Oblezin},\cite[p. 22]{Oblezin1}.
\bigskip

\begin{theorem}\label{T:complimentary1}
 Let $\underline{\nabla}$ be as defined above. Then its residue at $p$ has two complementary invariant subspaces:
\[
	\emph{Ker}(\mathrm{Res}_p\underline{\nabla}-(\lambda+1) I)\quad\mbox{ and }\quad W.
\]
\end{theorem}
\medskip

\begin{proof}
Let $(U,x)$ be a local chart centered around $p$ and $s_1,\ldots,s_n\in\mathcal{E}(U)$ are linearly
independent sections so that $W=\mbox{span}\{s_{{k+1},p},\ldots,s_{n,p}\}$. Now $xs_1,\, xs_2,\cdots,\, xs_k,\, s_{k+1},\cdots, s_n\in\underline{\mathcal{E}}(U)$
are linearly independent sections and
\begin{eqnarray*}
\underline{\nabla}(xs_j)&=&x\underline{\nabla}s_j+s_j\otimes dx\\
                       &=&xs_j\otimes\dfrac{\lambda dx}{x}+xs_j\otimes\dfrac{dx}{x} \quad\mod\mathfrak{m}_p\mathcal{E}_p\\
                       &=&(\lambda+1)\, xs_j\otimes\dfrac{dx}{x}\quad\mod\mathfrak{m}_p\mathcal{E}_p,
\end{eqnarray*}holds for each $1\le j\le k$. It is easy to show that $W$ is also invariant under the residue of $\underline{\nabla}$ at $p$. Indeed, let {$s_j\in \underline{\mathcal{E}}(U)\ (j=k+1,\cdots,n$)}. 
Then
	\[
		\underline{\nabla}s_j ={\nabla}s_j\\
		=(\mbox{Res}_p\, \nabla)\, s_j\otimes \frac{dx}{x} \mod\mathcal{E}_p.
	\]
Hence $(\mathrm{Res}_p\underline{\nabla})\, s_j= (\mathrm{Res}_p{\nabla})\, s_j\in W$.
\end{proof}
\bigskip

Verses the (\textit{lower}) modification of bundles introduced above, there is another type of modification called \textit{upper modification}. We will only give a brief description here.

\bigskip

\begin{definition}
Let $X$ be a Riemann surface. Given a sheaf of $\mathcal{O}_X-$modules $\mathcal{E}$ and $p\in X$, $W\subset\mathcal{E}_p$. Denote the (lower) modification of $\mathcal{E}$ at $(p,W)$ by $\underline{\mathcal{E}}$. 
	\begin{enumerate}
		\item We define the \textit{upper modification}  $\overline{\mathcal{E}}$ of $\mathcal{E}$ at $(p,W)$ to be
		\[
			\overline{\mathcal{E}}=\underline{\mathcal{E}}\otimes\mathcal{O}(p),
		\]
		where $\mathcal{O}(p)$ is the divisor line bundle of $1\cdot p$;
	\item Moreover, if $(\mathcal{E},\nabla)$ is a flat bundle with a log singularity at $p$ and $\mathcal{E}_p$ is decomposed into complementary subspaces
\[
\mathrm{Ker}(\mathrm{Res}_p\nabla-\lambda I)\quad \mbox{ and }\quad V,
\]
we denote the upper modification of $\mathcal{E}$ at $(p,\mathrm{Ker}(\mathrm{Res}_p\nabla-\lambda I))$ by $\overline{\mathcal{E}}$.
Then $\overline{\mathcal{E}}$ is equipped with a flat connection $\overline{\nabla}$ as in Definition \ref{D:modified_nabla},
such that there is a morphism  $(\mathcal{E},\nabla)\longrightarrow(\overline{\mathcal{E}},\, \overline{\nabla})$.
	\end{enumerate}
\end{definition}
\bigskip

The following result can be derived similar to that of Theorem \ref{T:complimentary1}.
\medskip

\begin{theorem}\label{T:complimentary2} Let $(\overline{\mathcal{E}},\, \overline{\nabla})$ be defined above. Then its residue at $p$ has $\overline{\nabla}-$invariant complementary subspaces
\[
\mathrm{Ker}(\mathrm{Res}_p\overline{\nabla}-(\lambda-1)I)\quad \mbox{ and }\quad V.
\]
\end{theorem}
\bigskip

Finally, one easily deduces the following general result about upper and lower modifications.
\medskip

\begin{lemma}[\cite{Oblezin}]
\label{T:upperlower}
Let $(\mathcal{E},\nabla)$ be a flat bundle with a log singularity at $p$, and $V$, $W$ are two complementary subspaces of $\mathcal{E}_p$ invariant under $\mathrm{Res}_p\nabla$. Denote the (lower) modification of $(\mathcal{E},\nabla)$ at $(p,W)$ by $(\underline{\mathcal{E}},\underline{\nabla})$. Then the upper modification of $(\underline{\mathcal{E}},\underline{\nabla})$ at $(p,V)$ is $(\mathcal{E},\nabla)$.
\end{lemma}

\section{Hypergeometric connections revisited} \label{S:hypergeo}

In order to better illustrate our monodromy approach to the main results later in this article, this section is reserved to review monodromy reduction of the hypergeometric connection in a sheave theoretic language that suits our purpose. Since we cannot find a  reference for the material that we revise, so we shall start with a reformulation of monodromy reduction of the classical hypergeometric equation entirely from monodromy consideration. Note that differential Galois theory works equally well for this purpose (see e.g., \cite{Kimura}), but the current setup is more appropriate for our purpose.

Solving the hypergeometric equation in a global sense has long been a difficult task. Therefore, special solutions are  usually considered. The following theorem is classical.
\medskip

\begin{theorem}\label{jacobi}
 Consider the hypergeometric equation
	\begin{equation}
		\label{E:hypergeometric}
	x(1-x)\dfrac{d^2y}{dx^2}+[c-(a+b+1)x]\dfrac{dy}{dx}-aby=0,
	\end{equation}
where $a$, $b$, $c\in\C$.

\begin{itemize}
\item
If $a\in\N$ {(resp. $b\in \N$)}, then the hypergeometric equation has a solution of the form $x^{1-c}(x-1)^{c-a-b}p(x)$ where $p$ is a polynomial of degree at most $a-1$ {(resp. $b-1$ at most)}.
\item
If $-a\in\N$  {(resp. $-b\in \N$)}, then the hypergeometric equation has a polynomial solution of degree at most $-a$ {(resp. $-b$)}.
\end{itemize}
\end{theorem}
\medskip

We require the following lemma which can be found in Beukers \cite[Lemma 3.9]{beukers2007}.

\begin{lemma}\label{hypergeometricreducible}
Let $M$, $N$ and $MN$ be $2\times 2$  matrices  each with distinct eigenvalues. If, however $1$ is a common eigenvalue of $M$, $N$ and $MN$, then $M$, $N$ and $MN$ share a common eigenvector.
\end{lemma}
\medskip

\subsubsection*{Proof of Theorem \ref{jacobi}} The proof
 is described in numerous classical literature (e.g. \cite[\S23, p. 90]{Poole}). Our objective now is to revise this proof via a geometric approach.

The monodromy representation of the hypergeometric equation above will be our most powerful tool. Let $M$, $N$ be the monodromy matrices of the standard hypergeometric equation relative to a certain basis. It is standard that
\[
\begin{array}{ll}
M&\mbox{ has eigenvalues }1, e^{2\pi i(1-c)}\\
N&\mbox{ has eigenvalues }1, e^{2\pi i(c-a-b)}\\
MN&\mbox{ has eigenvalues }e^{-2\pi ia}, e^{-2\pi ib}.
\end{array}
\]
Now $a\in\mathbb{Z}$. Let $v$ be a common eigenvector (which represents a solution $f$ of the hypergeometric equation defined on a small open set) to the matrices 
$M$, $N$ and $MN$ that is guaranteed by the  Lemma \ref{hypergeometricreducible}. Then after some routine consideration together with the Fuchsian condition, only two out of the total eight possibilities remain, which are, either
\[
\begin{array}{lll}
Mv=v&\mbox{ and }Nv=v&\mbox{ and }MNv=v\\
\mbox{or}&&\\
Mv=e^{2\pi i(1-c)}v&\mbox{ and }Nv=e^{2\pi i(c-b)}v&\mbox{ and }MNv=e^{-2\pi ib}v\\
\end{array}
\]
In case $a\in\mathbb{N}$
 and assume that we are in the former case. The locally defined function $f$ has trivial monodromy, and hence extends to a rational function with poles possibly at $0$, $1$ and $\infty$. However, $0$ is a local exponent at both $0$ and $1$, and {$a\in\mathbb{N}$} is a local exponent at $\infty$. Therefore, $f$ is an analytic function defined on $\mathbb{P}^1$ which has at least a zero {but no pole}, which is impossible. So we conclude that when $a\in\mathbb{N}$, {only} the latter case above is {possible}. Then we consider the locally defined function $x^{c-1}(x-1)^{a+b-c}f$ {which} has trivial monodromy and thus extends to a rational function. By a similar analysis of local exponents, $x^{c-1}(x-1)^{a+b-c}f$ is indeed a polynomial of degree at most $a-1$. This completes the proof of the first part. In the remaining case when $-a\in\N$ one reverses the roles in the analysis of the above two cases. Finally, if $a=0$, then one deduces that the hypergeometric equation admits a constant solution. This completes the second part of the proof.
\qed
\medskip

Observe that the proof above is highly dependent  on the monodromy of the hypergeometric operator, the same argument works for an operator with the same monodromy. The natural geometric object encoding the information of monodromy is a flat connection.
\bigskip

Let $X=\mathbb{P}^1\backslash\{0,1,\infty\}$, and $A_0$, $A_1$ are $2\times 2$ matrices with complex entries. Consider the connection $\nabla$ in $\mathcal{O}_X\oplus\mathcal{O}_X$ whose matrix relative to the canonical basis is
\[
-\Big[\dfrac{A_0}{x}+\dfrac{A_1}{x-1}\Big]dx,
\]
suppose further that $A_0$ has eigenvalues $0$ and $1-c$, $A_1$ has eigenvalues $0$ and $c-a-b-1$, $A_0+A_1$ has eigenvalues $-a$, $-b$. If $(y_1,y_2)^T$ is a local horizontal section of $\nabla$, then $y_1$ satisfies
\[
\dfrac{d^2y_1}{dx^2}+\Big[\dfrac{c}{x}+\dfrac{2-c+a+b}{x-1}-\dfrac{1}{x-\lambda}\Big]\dfrac{dy_1}{dx}+\dfrac{abx-\mu}{x(x-1)(x-\lambda)}y_1=0,
\]
where $\mu$ is a constant depending on the residue matrices, $\lambda$ {is an apparent singularity and it is also} the zero of the ($1,2)-$entry of the matrix ${A_0}/{x}+{A_1}/{(x-1)}$. Although the classical hypergeometric equation is obtained only when the choice $\lambda=1$ is made, that is
\[
A_0=\left(\begin{array}{cc}0&b\\0&1-c\end{array}\right)\mbox{ and }A_1=\left(\begin{array}{cc}0&0\\-a&c-a-b-1\end{array}\right),
\]
other choices of $\lambda$ should yield the same type of equations which are studied collectively by investigating $\nabla$.

Now {an analogue of} Theorem \ref{jacobi}
is rewritten as 
\begin{theorem}\label{T:subbundle}
Let $X=\mathbb{P}^1\backslash\{0,1,\infty\}$, $A_0$ is a $2\times 2$ matrix with eigenvalues $0$ and $\mu\neq 0$; $A_1$ is a $2\times 2$ matrix with eigenvalues $0$, $\nu\neq 0$, $\nabla$ is the connection in $\mathcal{O}_X\oplus\mathcal{O}_X$ 
whose matrix relative to the standard basis is $-\big[\dfrac{A_0}{x}+\dfrac{A_1}{x-1}\big]\, dx$.
\begin{itemize}
\item[(i)]
If $m\in\mathbb{N}$ is an eigenvalue of $A_0+A_1$, then there exists a non-trivial morphism of flat bundles\\ $(\mathcal{O}_X,d)\to(\mathcal{O}_X\oplus\mathcal{O}_X,\nabla)$.
\item[(ii)]
If $-m\in\mathbb{N}$ is an eigenvalue of $A_0+A_1$, then there exists a non-trivial morphism of flat bundles $\big(\mathcal{O}_X,d-\mu\dfrac{dx}{x}-\nu\dfrac{dx}{x-1}\big)\to(\mathcal{O}_X\oplus\mathcal{O}_X,\nabla)$.
\end{itemize}
\end{theorem}
\medskip

\begin{proof} Under the hypothesis of the first part of Theorem \ref{T:subbundle}, we see that the hypergeometric equation \eqref{E:hypergeometric} admits a global solution which, is a polynomial, by Theorem \ref{jacobi}, so is the first component of a $\nabla$ horizontal section. A straightforward inspection of the second component of the same   $\nabla-$horizontal section reveals that it is also a polynomial. Consequently, a non-trivial global $\nabla-$horizontal section $f$ exists.  The $\mathcal{O}_X$-linear map
$\mathcal{O}_X\to\mathcal{O}_X\oplus\mathcal{O}_X$ defined as multiplication by $f$ fits into the following  commutative diagram for each open set $U$
	\begin{equation*}
		\begin{array}{ccc}
			\mathcal{O}(U) &{d\atop \xrightarrow{\hspace*{.6cm}}}  &\Omega (U)\\
			\times f\left\downarrow\rule{0cm}{0.4cm}\right.\phantom{\omega}
			& &\phantom{\omega}\left\downarrow\rule{0cm}{0.4cm}\right.\times f\\
			\mathcal{O}(U)\oplus\mathcal{O}(U)& {\nabla \atop \xrightarrow{\hspace*{.6cm}}} & \Omega (U)\oplus
			\Omega (U)
		\end{array}.
    \end{equation*}
The second part of Theorem \ref{T:subbundle} can be proved similarly.
\end{proof}

Later in Section \ref{S:heunpolyn}, we will study the analogous phenomenon in the case of connections with four log singularities.

\subsection{Kummer symmetry: sheave theoretic interpretation}
In order to better illustrate our description of the Heun equation below, we revisit
 the symmetry of the solutions of  the hypergeometric equation (Kummer symmetry)  in a sheaf theoretic language. In particular this generates more criteria for special solutions other than those derived from Theorem \ref{jacobi}  and Theorem \ref{T:subbundle} in this subsection. Of course, the result is well-known in numerous classic texts, see e.g., \cite{Poole}.
 A more comprehensive investigation of symmetry would involve sympletomorphisms between moduli spaces of local systems, (see e.g. Oblezin \cite{Oblezin}) which is beyond the scope of this article.
 
  Loosely speaking,
the group of Kummer symmetry is generated by two parts: transformations of the independent variable and transformations of the dependent
variable.

Given a M\"obius transformation $f$ mapping $\{0,1,\infty\}$ into $\{0,1,\infty\}$,
and a local system $\mathcal{L}$ which is the
sheaf of horizontal sections of a connection in a rank-two bundle over
$\mathbb{P}^1\backslash\{0,1,\infty\}$ with log singularities at $0$, $1$ and
$\infty$, the push-forward $f_*\mathcal{L}$
is also one with
three log singularities at $0$, $1$, $\infty$. Thus, the group of such M\"obius transformations acts on the set of such local systems.

We also have the following actions on local systems
\[
\begin{array}{l}
S_0:\mathcal{L}\longmapsto\mathcal{L}\otimes\mathcal{K}_0^{-\gamma}\\
S_1:\mathcal{L}\longmapsto\mathcal{L}\otimes\mathcal{K}_1^{-\delta},
\end{array}
\]
where $A_0$ has eigenvalues $0$ and $\gamma$, $A_1$ has eigenvalues $0$ and $\delta$.
These actions take the sheaf of horizontal sections of a connection with three singularities $0$, $1$, $\infty$ to the sheaf of horizontal sections
of another connection with three singularities $0$, $1$, $\infty$. The actions $S_0$, $S_1$ generate a group which acts on the ``dependent variable''
of a hypergeometric equation.
These actions 
generate the classical Kummer symmetry group. They send local systems of
certain special types to other local systems of the same type. In other words,
special solutions of the Gauss hypergeometric equation are sent to the same
kind of special solutions via the Kummer symmetry.

We easily identify the Kummer symmetry from the discussion above in the following well-known theorem.
\begin{theorem}  The group generated by $S_0,\, S_1$ together with the push-forwards $f_\ast$  of those M\"obius transformations $f$ preserving $\{0,\, 1,\,\infty\}$ is isomorphic to the  group of signed permutations of three letters module $\mathbb{Z}_2$.
\end{theorem}

In general, the construction of the Kummer symmetry group
generalises to local systems over $\mathbb{P}^1$ with more than three log
singularities. This generalisation is elaborated in \cite{Maier1} (or in \cite{CCT} when $\mathbb{P}^1$ is replaced by a complex torus with four log singularities).

We can obtain other special solutions analogous to those in Theorem \ref{jacobi} by applying the actions described  by the above symmetry group.  The result is summarised in the following well-known statement (see Poole \cite[\S23,  p. 90]{Poole}).

\begin{proposition}
If there exist $\lambda\in\{0,1-c\}$, $\mu\in\{0,c-a-b\}$ and $\nu\in\{a,b\}$ such that $\lambda+\mu+\nu\in\mathbb{Z}$, then the hypergeometric equation \eqref{E:hypergeometric} is reducible.
\end{proposition}

\section{Simultaneous diagonalisation} \label{S:heun}

We consider the connection \eqref{E:Heun_connection} over a rank-2 vector bundle with four regular singularities $\{0,\, 1,\, a,\, \infty\}$ in $\mathbb{P}^1$. Its matrix relative to a frame is
    \begin{equation}\label{E:heun-scheme}
		-\Big[\dfrac{A_0}{x}+\dfrac{A_1}{x-1}+\dfrac{A_a}{x-a}\Big]dx,
	\end{equation}
which we call a \textit{Heun-type connection} as in the Introduction.

In this setup, an appropriate choice of the accessory parameter $q$ means the simultaneous diagonalisation of two matrices amongst the residues at $\{0,\, 1,\, a,\, \infty\}$. So we denote
\bigskip

\begin{definition}\label{D:WGRM}
\begin{equation*}\label{G}\tag{\textbf{WGRM}}
	\Big\{ 
	\begin{array}{l}
\mbox{weak global reducible condition of monodromy:}\\
A_a\mbox{ and }A_{\infty}\mbox{ are simultaneously diagonalisable}.
\end{array}
	\Big\}
\end{equation*}
\end{definition}
\noindent for later applications.
\bigskip 

\subsection{Erd\'elyi's expansions revisited}
In order to illustrate under what circumstance in some classical consideration  of Fuchsian equations which are equivalent to the 
(\textbf{WGRM}) of Fuchsian connections defined above,
we review Erd\'elyi's infinite expansions in terms of hypergeometric functions,  which are used to study the monodromy group of the Heun equation
\begin{equation}\label{E:heun_2}
         \frac{d^2y}{dx^2}+
         \Big(\frac{\gamma}{x}+\frac{\delta}{x-1}+\frac{\epsilon}{x-a}\Big)\frac{dy}{dx}+
         \frac{\alpha\beta x-q}{x(x-1)(x-a)}y=0.
\end{equation}
Note that the parameters satisfy the Fuchsian constraint $\alpha+\beta-\gamma-\delta-\epsilon+1=0$ in \cite{Erdelyi1,Erdelyi2}.

The region of convergence of such infinite expansions 
can be used as a measure of the difference between the monodromy group of the hypergeometric equation and that of the Heun equation. For example, Erd\'{e}lyi \cite[(4.2)]{Erdelyi1} (1942) represented the local Heun function $Hl(a,q;\alpha,\beta,\gamma,\delta;x)$ (that is, the local analytic solution at the singularity $x=0$) by the hypergeometric function series $\displaystyle\sum_{m=0}^\infty X_m \varphi_m^1(x)$, where
	\[
		\varphi_m^1(x):=
\frac{\Gamma(\alpha-\delta+m+1)\Gamma(\beta-\delta+m+1)}
{\Gamma(\alpha+\beta-\delta+2m+1)}
x^m\cdot
\sideset{_2}{_1}{\operatorname{F}}\left({\begin{matrix}
                 \alpha+m,\beta+m\\
                  \alpha+\beta-\delta+2m+1
                 \end{matrix}};x\right)
	\]and the coefficients $X_m$ satisfy a three-term recursion \eqref{E:Three_term} given in the Appendix \ref{A:erdelyi}. If the \textit{accessory parameter} $q$ does not satisfy the \textit{infinite continued fraction} \eqref{E:cf}, then the infinite sum converges in the bounded region $\Omega_0$  defined in \eqref{E:omega_0} which contains $x=0$ in its interior but excludes $x=1$. If, however, that the accessory parameter $q$ \textit{satisfies} the {infinite continued fraction} \eqref{E:cf}, then the infinite sum converges in the larger region $\Omega_1$ defined in \eqref{E:omega_1} which equals to $\mathbb{C}$ with a branch cut from $1$ to $\infty$; see also the Remark after Theorem \ref{T:convergent}. 
A second infinite hypergeometric type solution $\displaystyle\sum_{m=0}^\infty X_m\varphi_m$ to the Heun equation linearly independent from the above infinite sum  $\displaystyle\sum_{m=0}^\infty X_m \varphi_m^1(x)$, where  each $\varphi_m$ can be \textit{any} linear combination of $\varphi_m^2,\cdots,\varphi_m^6$ defined in \cite[(4.2)]{Erdelyi1} was also derived by Erd\'elyi. If the coefficients $\{X_m\}$ defined by \eqref{E:Three_term} and the accessory parameter $q$
satisfies the  {infinite continued fraction} \eqref{E:cf}, then the infinite sum converges in the region ${\Omega}^-_1$ with non-empty intersection $\Omega_1$ defined in \eqref{E:omega_1}. When the accessory parameter $q$ satisfies the  {infinite continued fraction} \eqref{E:cf}, the region  $\Omega_1\cap {\Omega}^-_1$ contains $x=1,\, \infty$ in its interior but excludes $x=0$, the condition (\textbf{WGRM}) defined above (with $x=a$ replaced by $x=1$) describes precisely that both the hypergeometric expansions $\displaystyle\sum_{m=0}^\infty X_m\varphi^1_m$ and $\displaystyle\sum_{m=0}^\infty X_m\varphi_m$ converge simultaneously in the common region $\Omega_1\cap {\Omega}^-_1$ that contains $x=1,\,\infty$.
\bigskip

\begin{remark}
The corresponding series solutions to Darboux equations can be found in \cite{CCT2}.
\end{remark}
\bigskip

One may further ask what happens if the domain of convergence of such a series includes $0$, $1$ and $a$. If this is the case, {then} $A_0$, $A_1$ and $A_a$ (and hence $A_{\infty}=-A_0-A_1-A_a$) are simultaneously diagonalisable. Then there are two line bundles invariant
under the connection $\nabla$ of Heun type scheme \eqref{E:heun-scheme}. That is, this classical Heun operator from \eqref{E:heun} is factorized into two commuting first order operators, which is an
uninteresting circumstance from our viewpoint in this paper. In general, we have to settle if the  series (\ref{E:2F1h}) converges in a bigger domain including both $0$ and either $1$ or $a$.

However, if the parameters (i.e., local monodromies) of a Heun operator are special, we may ask if solutions of more special types exist. This is the study of the global properties of Heun-type
connections which we will carry out in the upcoming sections.

\section{Type I degeneration: One singularity becomes apparent} \label{S:eigenspaceI}
\label{apparentsingularity}

\subsection{Resolving singularities}
\begin{theorem} \label{term1}
If $\epsilon\in\N$ in equation \eqref{E:heun}, then there exists $q$ such that the series solution (\ref{E:2F1h}) terminates.
\end{theorem}

\begin{proof}
The proof follows from three-terms recurrence relation given in \cite{Erdelyi1}.
\end{proof}
\bigskip

The theorem suggests that with the special parameters $\epsilon,\, q$ mentioned in the theorem above, the monodromy of such a special Heun equation reduces to that of a
hypergeometric equation, and hence it has a local solution written in terms of the hypergeometric functions. This suggests that the singularity $x=a$ is \textit{removed}.

For the sake of convenience, we will name {the} additional condition on $A_a$ in \eqref{E:heun-scheme} which characterizes such a Heun-type connection:
\begin{definition}\label{D:WAS}
\begin{equation*}\tag{$\mathbf{WAS}(m)$}
	\Big\{
	\begin{array}{l}
\mbox{weak apparent singularity condition:}\\
0\mbox{ and }m\in\mathbb{N}\mbox{ are the eigenvalues of }A_a
\end{array}
	\Big\}
\end{equation*}
\end{definition}
The name ``weak apparent singularity condition" suggests that it is somewhat different from the ``apparent singularity condition" one usually sees in the theory of Fuchsian differential equations. The common notion of ``apparent singularity at $a$" means that a second order Fuchsian differential equation $Ly=0$ (I) has the difference of local exponents at $a$ being an integer; and (II) an accessory parameter is chosen appropriately so that the local monodromy of $L$ at $a$ is diagonalisable (i.e., only non-logarithmic local solutions as in \cite[pp. 69-70]{Poole}). Here we only focus on the criterion (I) and build a function space invariant under $L$ as mentioned in \eqref{E:eigenvalue_prob_1} and \eqref{E:eigenvalue_prob_2}. Since these function spaces depend on the domains chosen, we will 
glue them together in order to build a sheaf of $\mathcal{O}-$modules. Being invariant under $L$ would mean that this sheaf is equipped with a flat connection. The detail will be given in Theorem \ref{T:resolving_sing} below.

\medskip

We also recall that if a divisor on $\mathbb{P}^1$ is given by $D=m(a)$, then its divisor line bundle is denoted by $\mathcal{O}(m(a))$ and the space of global sections of $\mathcal{O}(m(a))$ is identified as
\[
	\mbox{span}\Big\{\dfrac{1}{(x-a)^{n}}:\ \ n=0,1,\cdots m\Big\}.
\]

The main result in this paper is the following theorem,
which says that if ($\mathbf{WAS}(m)$) is satisfied, then the Heun type connection can be interpreted as
one without any singular points away from $0$, $1$ and $\infty$, provided that the underlying vector bundle
is chosen carefully. Hence the flat bundle $(\mathcal{E},\nabla)$ in the theorem below has only three singular points. Thus, it behaves like a connection of hypergeometric type. We first deal with the case when $m\ge 1$ in this subsection before some applications in the next subsection. The general case where $m\in\mathbb{Z}$ and when the apparent singularity being any one of $\{0,\, 1,\, \infty\}$ other than $a$ will be given in the last subsection within this section.
\bigskip

\begin{theorem}[(Resolving singularity - version I)]
\label{T:resolving_sing}

Let $X=\mathbb{P}^1\backslash\{0,1,\infty\}$ and
$a\in X$, and
let $A_0,A_1,A_a$ be $2\times 2$ matrices with complex entries such that
$0$ and $m\in\mathbb{N}$ are the eigenvalues of $A_a$.
Then there exist a sheaf of locally free $\mathcal{O}_X$-module
$\mathcal{E}$, and a flat connection
$\nabla:\mathcal{E}\longrightarrow \mathcal{E}\otimes\Omega_X$ such that the connection matrix of $\nabla$ relative to a frame over an open subset of $X$ \emph{not containing} $a$ is given by
\[
-\Big[\dfrac{A_0}{x}+\dfrac{A_1}{x-1}+\dfrac{A_{a}}{x-{a}}\Big]\, dx.
\]
\end{theorem}
\bigskip

\begin{proof}
Let $\mathcal{F}$ be the sheaf of horizontal sections of the connection $\overline{\nabla}$
in the trivial bundle $\mathcal{O}_X\oplus\mathcal{O}_X$ whose connection matrix relative to the
canonical basis is
	\[
		-\Big[\dfrac{A_0+A_{a}}{x}+\dfrac{A_1}{x-1}\Big]\,dx.
\]
Define a connection $\nabla$ in $\mathcal{F}\otimes\mathcal{O}(m(a))$ by
	\[
		\nabla(fs)=s\, df+\Big[\dfrac{a(A_{a}s)}{x(x-a)}\Big]f\, dx,
	\]
for every open set
$U\subset X\backslash\{a\}$, $s\in\mathcal{F}(U)$ and analytic function $f$
with divisor $\geq -m(a)$.

It is clear that $\nabla$ can possibly have a log singularity at
$a$. We will construct a subsheaf of $\mathcal{F}\otimes\mathcal{O}(m(a))$ such that the restriction of  $\nabla$ to it does not contain any singularities.

Let $v=x(x-1)\frac{d}{dx}$ be a fixed vector field and
	\[
		\mathcal{E}(U)=\mbox{the }\nabla_v-\mbox{module generated by }\mathcal{F}(U).
	\]

\noindent The issue here is to show that $\mathcal{E}(U)\subset \mathcal{F}\otimes\mathcal{O}(m)(U)$.
Observe that for each local section $s\in\mathcal{F}(U)$,
\begin{eqnarray*}
\nabla_vs
&=&\dfrac{ax(x-1)}{x(x-a)}(A_{a}s)\\
&=&\dfrac{a(a-1)}{x-a}(A_{a}s)+aA_{a}s.
\end{eqnarray*}
Thus,
	\[
		(x-a)\nabla_v s \in \mathrm{Image\ of\ }A_{a} \mod \mathfrak{m}_{a},
	\]
and hence by induction
\[
		(x-a)^m(\nabla_v)^m s \in \mathrm{Image\ of\ }A_{a} \mod \mathfrak{m}_{a},
	\]

Hence we write
\[
(\nabla_v)^{m} s=\dfrac{A_{a}t}{(x-a)^{m}}+\mbox{(lower order terms)},
\]
for some $t\in\mathcal{F}(U)$. 
Now,
\begin{eqnarray*}
(\nabla_v)^{m+1}s
&=&-m\dfrac{x(x-1)}{(x-a)^{m+1}}(A_{a}t)+a\dfrac{x-1}{(x-a)^{m+1}}(A^2_{a}t)+\mbox{(lower order terms)}\\
&=&-m\dfrac{a(a-1)}{(x-a)^{m+1}}(A_{a}t)+a\dfrac{a-1}{(x-a)^{m+1}}(A^2_{a}t)+\mbox{(lower order terms)}\\
&=&\dfrac{a(a-1)}{(x-a)^{m+1}}(A^2_{a}-mA_{a})t+\mbox{(lower order terms)}\\
&\in&\mathcal{F}\otimes\mathcal{O}(m)(U).
\end{eqnarray*}

\noindent

Therefore, the restriction of  $\nabla$ to  the subsheaf $\mathcal{E}$ of $\mathcal{F}\otimes\mathcal{O}(m(a))$  no longer contains the singularity $a$. It is now clear that the flat bundle $(\mathcal{E},\nabla)$ has the prescribed connection matrix relative to the canonical basis of $\mathcal{O}_X\oplus\mathcal{O}_X$.
\end{proof}
\bigskip

\subsection{Applications}\label{S:applications}

We first show how to derive Erd\'elyi's expansion \eqref{E:erdelyi} from the sheave of horizontal sections (local system) $\mathcal{F}$ of the connection
	\begin{equation}\label{E:representation_1942}
		\overline{\nabla}:=d-\Big[\dfrac{A_0+A_{a}}{x}+\dfrac{A_1}{x-1}\Big]\, dx
	\end{equation}
that was used in the proof of Theorem \ref{T:resolving_sing}. The derivation of expansion \eqref{E:erdelyi} from this construction will be given in Example \ref{Eg:1} below. The idea behind the construction \eqref{E:representation_1942} is to displace the monodromy information at the singularity $a$  to the origin $x=0$ while the $x=a$ is being resolved.
Then we explore how to obtain other expansions, including some new ones,  for solutions of the Heun equation \eqref{E:heun} by choosing different forms of $\overline{\nabla}$. In fact, we not only recover all the hypergeometric type expansion solutions derived by Erd\'elyi in \cite{Erdelyi1, Erdelyi2}, but also exhibit a new expansion \eqref{E:new-expansion} below, amongst a large number of possible hypergeometric expansions of local solutions for the Heun equation below from the general theory we propose here.
\bigskip

\begin{example} \label{Eg:1}
Consider the Fuchsian connection \eqref{E:representation_1942}
		\[
			B_0:=A_0+A_a=-A_1-A_\infty,
		\]
		where $B_0$ denotes the residue matrix of
		$\overline{\nabla}$ at $x=0$.
		\[
		\begin{split}
			\Tr(B_0)=\Tr(A_0+A_a) &=-\Tr(A_1)-\Tr(A_\infty)\\
			&=(\delta-1)-\alpha-\beta\\
			&=\delta-\alpha-\beta-1-n+n.\\
		\end{split}
		\]
		It follows from Lemma \ref{L:fuchsian_relation} (with $n=2$) that  $\Tr(B_0)$ differs from the sum of the two indicial roots at $x=0$ of the scalar equation satisfied by the first component by one. That is, we increase the $\delta-\alpha-\beta-1-n$  by one when it is written in the classical Riemann scheme. The analysis implies that we have formal sum of schemes
		\begin{equation}\label{E:sum_1942pm}
		\sum_n c_n\, P
			\left(
			\begin{matrix}
			0                     & 1        & \infty\\
			\, n                     & 0        & \alpha \\
			\delta-\alpha-\beta- n & 1-\delta & \beta
			\end{matrix}
			;\ x
			\right)
	\end{equation} which recovers the scheme given in \cite{Erdelyi1}. One can derive the formal sum
		\begin{equation}\label{E:erdelyi_pm}
		\sum_{n=0}^\infty \mathcal{C}_n\,
x^{n}
\sideset{_2}{_1}{\operatorname{F}}\left({\begin{matrix}
                 \alpha+n,\ \beta+n\\
                  \alpha+\beta-\delta+ 2n+1
                 \end{matrix}};x\right)
		\end{equation}
which is precisely the infinite sum mentioned in \eqref{E:erdelyi}, where the coefficients $\mathcal{C}_n$ satisfy a certain three-term recursion \cite[(5.3), (5.4)]{Erdelyi1} (also see Appendix~\ref{A:erdelyi}).
\end{example}	
\bigskip

\begin{example}[\cite{Erdelyi2}]\label{E:bad_infinty} Consider the sheave of horizontal sections (local system) $\mathcal{F}$ of the connection 
	\begin{equation}\label{E:representation_1944}
\overline{\nabla}=d-\Big[\dfrac{A_0}{x}+\dfrac{A_1}{x-1}\Big]dx,
	\end{equation}
from which one defines
		\[
			\nabla(fs)=s\, df+\Big[\dfrac{A_{a}s}{x-a}\Big]f\, dx,
		\]
	for every open set $U\subset X\backslash\{a\}$, $s\in\mathcal{F}(U)$ and analytic function $f$
with divisor $\geq -m(a)$.
Then
		\[
			{B}_\infty=A_a+A_\infty=-A_0-A_1
		\]
		is the residue matrix of $\overline{\nabla}$ at $\infty$.
Thus
	\[
		\begin{split}
			\Tr({B}_\infty)=\Tr(A_a+A_\infty)
			&:=\lambda+\mu\\
			&=(\lambda+n)+(\mu-n),\
		\end{split}
	\]where $\lambda,\, \mu$ are the two exponents for the corresponding hypergeometric connection at the singularity $x=\infty$ and $n\in \mathbb{Z}$ is arbitrary (see Theorem \ref{T:resolving_sing}). This approach would give rise to the expansion that was obtained  by Erd\'elyi in another paper \cite{Erdelyi2} in 1944. That is,
	\begin{equation}\label{E:sum_1944}
		\sum_n c_n\, P
			\left\{
			\begin{matrix}
			0                     & 1        & \infty\\
			0                     & 0        & \lambda+m \\
			1-\gamma & 1-\delta   & \mu-m
			\end{matrix}
			;\ x
			\right\},
	\end{equation} 
which gives the expansion \eqref{E:erdelyi_2} for the local solution at the origin with exponent 0.	
We note that the adjustment  of the Fuchsian relations of transition from a connection form to scalar differential equation form as described in Lemma \ref{L:fuchsian_relation} is implicitly absorbed in the choices of the notations $\lambda,\, \mu$ above already.
 \end{example}
\bigskip

\begin{example}[\cite{STW}]\label{E:theorists} Shiga, Tsutsui and Wolfart considered the transcendence of Schwarz maps at algebraic arguments from Fuchsian equations with the same monodromy of a hypergeometric equation with several apparent singularities. They showed that their differential equation admits a holomorphic solution $f$ at $x=0$ normalised with $f(0)=1$ can be expressed in the form 
	\begin{equation}\label{E:STW}
		f(x)=c_0(t)
		\sideset{_2}{_1}{\operatorname{F}}\left({\begin{matrix}
                 \mu^\prime,\ \mu^{\prime\prime}\\
                 1-m-\nu_0
                 \end{matrix}};\, x\right)
                 +\cdots+
                 c_m(t)
		\sideset{_2}{_1}{\operatorname{F}}\left({\begin{matrix}
                 \mu^\prime,\ \mu^{\prime\prime}\\
                 1-\nu_0
                 \end{matrix}};\, x\right)
	\end{equation}
such that $\sum_{k=0}^mc_k(t)=1$, where $t$ is determined by the locations of the apparent singularities. We note that while the locations of the apparent singularities are considered movable in \cite{STW}, the locations of our consideration here are fixed. 
In \cite{STW}, the Fuchsian equation has $m$ apparent singularities  where the exponent difference at each of these singularities is $2$. However, we can show directly that the \eqref{E:STW} can also be given from our \eqref{E:heun} where the exponent difference at the apparent singularity $x=a$ is the integer $m$, i.e., when (\textbf{WAS}$(m)$) holds. Application of the contiguous relation
	\[
		(c-a-1){}_2F_1(a,\, b;\, c;\, x)+a{}_2F_1(a+1,\, b;\, c;\, x)-(c-1){}_2F_1(a,\, b;\, c-1;\, x)=0
	\]
repeatedly to each term of \eqref{E:STW} yields
	\begin{equation}\label{E:STW2}
		f(x)=d_0(t)
		\sideset{_2}{_1}{\operatorname{F}}\left({\begin{matrix}
                 \mu^\prime+m,\ \mu^{\prime\prime}\\
                 1-\nu_0
                 \end{matrix}};\, x\right)
                 +\cdots+
                 d_m(t)
		\sideset{_2}{_1}{\operatorname{F}}\left({\begin{matrix}
                 \mu^\prime,\ \mu^{\prime\prime}\\
                 1-\nu_0
                 \end{matrix}};\, x\right)
	\end{equation}
	where $\sum_{k=0}^m d_k(t)=1$. A further application of the contiguous relation	
	\[
		(b-a){}_2F_1(a,\, b;\, c;\, x)+a{}_2F_1(a+1,\, b;\, c;\, x)-b{}_2F_1(a,\, b+1;\, c;\, x)=0
	\]
to \eqref{E:STW2} repeatedly yields
	 	\[
		\sum_{k=0}^m e_k(t)\,
\sideset{_2}{_1}{\operatorname{F}}\left({\begin{matrix}
                 \mu^\prime+k,\ \mu^{\prime\prime}-k\\
                  1-\nu_0
                 \end{matrix}};x\right),
          \]
which is a terminated form of \eqref{E:erdelyi_2} where $\lambda=\mu^\prime$, $\mu=\mu^{\prime\prime}+m$. 
\end{example}
\bigskip

 \begin{example} Suppose next that $\mathcal{F}$ is the sheave of horizontal sections (local system) of the connection 
	\begin{equation}\label{E:nabla-new}
		\overline{\nabla}=d-\Big[\dfrac{A_0}{x}+\dfrac{A_1+A_a}{x-1}\Big]dx,
	\end{equation}
from which another connection $\nabla$ is defined by
		\[
		\nabla(fs) := s\, df+\Big[\dfrac{(1-a)A_{a}s}{(x-1)(x-a)}\Big]f\, dx,
	\]
for every open set $U\subset X\backslash\{a\}$, $s\in\mathcal{F}(U)$ and analytic function $f$
with divisor $\geq -m(a)$. Then
		\[
			{B}_1=A_a+A_1=-A_0-A_\infty
		\]
		is the residue matrix of  $\overline{\nabla}$ at $x=1$. Hence 
	\[
		\begin{split}
			\Tr({B}_1)&=\Tr(A_1+A_a) =-\Tr(A_0)-\Tr(A_\infty)\\
			&= (\gamma-1)-\alpha-\beta\\
			&=(\gamma-1-\alpha-\beta-n)+n.\\
		\end{split}
	\]  Similar to the consideration in Example 5.4 that the indicial root $\gamma-\alpha-\beta-n$ of the corresponding classical hypergeometric equation is one bigger than the eigenvalue $\gamma-1-\alpha-\beta-n$ that appears above. Hence this  gives rise to 	\begin{equation}\label{E:sum_new}
		\sum_n c_n\, P
			\left\{
			\begin{matrix}
			0                     & 1                                    & \infty\\
			0                     & n                             & \alpha \\
			1-\gamma      & \gamma-\alpha-\beta- n & \beta
			\end{matrix}
			;\ x
			\right\},
	\end{equation}which contains expansions that were \textit{neither found in} \cite{Erdelyi1} \textit{nor in} \cite{Erdelyi2}.
 Thus, it is possible to obtain
	\begin{equation}\label{E:new-expansion}
		\sum_{n=0}^\infty \mathcal{C}_n\,
(x-1)^{n}
\sideset{_2}{_1}{\operatorname{F}}\left({\begin{matrix}
                 \alpha+ n,\ \beta+ n\\
                  \gamma
                 \end{matrix}};\, x\right)
		\end{equation}
where the coefficients $\mathcal{C}_n$ would satisfy a certain three-term recursion which we omit.
\end{example}
\bigskip

\begin{example}
	Indeed, by choosing
	\begin{equation}\label{E:nabla-3}
		\overline{\nabla^t}=d-\Big[\dfrac{A_0+tA_a}{x}+\dfrac{(1-t)A_a+A_1}{x-1}\Big]dx.
	\end{equation}
	Let $\mathcal{F}$ be the sheave of horizontal sections of $\overline{\nabla^t}$. Then one defines
	\[
		\nabla(fs):=s\, df+\Big[
		\frac{tA_a}{x}+	\dfrac{(t-1)A_{a}}{x-1}-\frac{A_a}{x-a}\Big]sf\, dx,
	\]
for every open set 
$U\subset X\backslash\{a\}$, $s\in\mathcal{F}(U)$ and analytic function $f$
with divisor $\geq -m(a)$. We see that the residues of $\overline{\nabla^t}$ at $x=0,\, 1$ are given by
	\[
		{B}_t:=(A_0+tA_a), \quad {C}_{t}:= ((1-t)A_a+A_1).
	\]
respectively, so that 
	\[
		{B}_t+{C}_{t}+A_\infty=0
	\]
is the Fuchsian relation for  $\overline{\nabla^t}$. Since the eigenvalues of $B_t,\ C_t$ are unknown, one cannot obtain Erd\'elyi type hypergeometric expansions from it in general. 		Notice that we recover the matrix representation \eqref{E:representation_1944} when $t=0$ and the representation \eqref{E:representation_1942} when $t=1$. 
	\end{example}
\medskip

\begin{remark}
 One observes from the discussion of above examples that  essentially all possible Erd\'elyi type hypergeometric expansions can be obtained this way.
\end{remark}

\subsection{Accessory parameters and invariant spaces}

In order to study the eigenvalue problems \eqref{E:eigenvalue_prob_1} and \eqref{E:eigenvalue_prob_2}, a $\nabla_v-$invariant function space is needed. 

As in the proof of Theorem \ref{T:resolving_sing}, let $\mathcal{F}$ be the sheaf of horizontal sections of a connection ${\nabla}$ with singularities $\{0,\, 1,\, \infty\}$ (we are allowed to choose different $\mathcal{F}$ as shown in the examples  in \S\ref{S:applications}).  Then given an open set $U$ and a section $s\in\mathcal{F}(U)$, either one of the spans 
	\[
		\mbox{span}\big\{A_as,\nabla_v A_as,(\nabla_v)^2 A_as,\cdots,(\nabla_v)^m{A_as}\big\},
	\]
or	\[
		\mbox{span}\Big\{A_as,\dfrac{A_as}{x-a},\dfrac{A_as}{(x-a)^2},\cdots,\dfrac{A_as}{(x-a)^m}\Big\},
	\]
is invariant under $\nabla_v$. So the eigenvalue problem \eqref{E:eigenvalue_prob_2} is well-defined. The matrix of $\nabla_v$ relative to the former basis is in rational canonical form, while the latter basis gives a
tridiagonal matrix. We choose to study our eigenvalue problem using the second basis. Let us illustrate our viewpoint in the following lower dimensional examples. In particular, we notice that the eigenvalues of these $\nabla_v$ are all integers.
\bigskip

\begin{example} \label{Eg1}
Given $m\in\N$, $k=1,2,\cdots,m$, we have
$A_a^2=mA_a$ and hence
	\[\nabla_v\left({A_as}\right)=\frac{ma(x-1)A_as}{x-a}=ma(a-1)\frac{A_as}{x-a}+ma\, A_as\]
and
\begin{eqnarray*}\nabla_v\left(\frac{A_as}{(x-a)^k}\right)&=&-\frac{kx(x-1)A_as}{(x-a)^{k+1}}+\frac{ma(a-1)A_as}{(x-a)^{k+1}}+\frac{maA_as}{(x-a)^k}\\
&=&a(a-1)(m-k)\frac{A_as}{(x-a)^{k+1}}+[(m-2k)a+k]\frac{A_as}{(x-a)^k}-k\frac{A_as}{(x-a)^{k-1}}.
\end{eqnarray*}
When $m=1$, the matrix representation of $\nabla_v$ relative to
		\[
			\Big\{A_as,\dfrac{A_as}{x-a}\Big\}
		\]
is
\[\begin{bmatrix}
a&-1\\
a(a-1)&1-a
\end{bmatrix},\]
and its eigenvalues are $0$ and $1$.
When $m=2$, the matrix representation of $\nabla_v$ relative to 	
	\[
		\Big\{A_as,\dfrac{A_as}{x-a},\dfrac{A_as}{(x-a)^2}\Big\}
	\]
is
	\[\begin{bmatrix}
		2a&-1&0\\
		2a(a-1)&1&-2\\
		0&a(a-1)&2-2a
	\end{bmatrix},\]
and its eigenvalues are $0$, $1$ and $2$. In general, for any $m\in\N$, the eigenvalues of $\nabla_v$
are $0,1,\cdots,m$.
\end{example}
\bigskip

The phenomenon of integral-valued (or equally-spaced)  eigenvalues for $\nabla_v$ exhibited in the Example \ref{Eg1} above can be explained by using the technique of bundle modifications as defined in Definition \ref{D:modification} (see \cite{Oblezin}).  {We summarise the observation in the following theorem.}
\bigskip

\begin{theorem}\label{T:eigenvalues_a}
Let $(\mathcal{E},\nabla)$
be a Heun type-connection such that the residue of $\nabla$ at the singular point $a$ has eigenvalues $0$, $-m$ for some $m\in\mathbb{N}$, , i.e., $(\mathbf{WAS}(m))$. Then there exists a vector field $w$ such that for each sufficiently small open set ${U}$, the set of eigenvalues of $\nabla_w:\ \mathcal{E}(U)\to \mathcal{E}(U)$ is given by $\{0,\, 1,\, 2,\, \cdots, m\}$, i.e.,  there exist $m+1$ local sections ${s}_0,{s}_1,\cdots,\, {s}_m\in\mathcal{E}(U)$ such that
	\[\nabla_w s_k=k\, {s}_k\]
for all $k\in \{0,\, 1,\, \cdots, m\}$.
\end{theorem}
\medskip

\begin{proof}
 Let $(\mathcal{E}_m,\nabla_m)$
be a connection of Heun type such that Res$_a\nabla$ has eigenvalues $0$, $-m\in-\N$.
Let its modification at $(a,\, $Ker(Res$_a\nabla))$ be
$(\mathcal{E}_{m-1},\nabla_{m-1})$. In particular, there exists a subsheaf
$(\mathcal{E}_{m-1},\nabla_{m-1})\to(\mathcal{E}_m,\nabla_m)$ such that Res$_a\nabla_{m-1}$ has eigenvalues
$0$, $-m+1$ by Theorem \ref{T:complimentary1}. 
Repeating this procedure yields a filtration of $(\mathcal{E}_m,\nabla_m)$:
\[
(\mathcal{E}_0,\nabla_0)\longrightarrow
\cdots \longrightarrow (\mathcal{E}_{m-1},\nabla_{m-1}) \longrightarrow (\mathcal{E}_m,\nabla_m).
\]
Now if $s$ is a local section of $\mathcal{E}_k$, then
\begin{eqnarray*}
\nabla_ks
&=&(\mbox{Res}_a\nabla_k)s\otimes \dfrac{dx}{x-a}\hspace{1cm}\mbox{mod }\Omega_a\\
&=&(\mbox{Res}_a\nabla_k+kI)s\otimes \dfrac{dx}{x-a}-ks\otimes \dfrac{dx}{x-a}\hspace{1cm}\mbox{mod }\Omega_a\\
&=& -ks\otimes\dfrac{dx}{x-a}\hspace{1cm}\mbox{mod (image(}\mathcal{E}_{k-1},\nabla_{k-1})\longrightarrow(\mathcal{E}_k,\nabla_k))\otimes\Omega_a
\end{eqnarray*}
{because of $\mathrm{Res}_a\nabla_k (\mathrm{Res}_a\nabla_k +kI)=0$ by Cayley-Hamilton}. 
Thus if $w=-(x-a)\dfrac{d}{dx}$, then
\[
(\nabla_{k})_ws=ks\hspace{1cm}\mbox{mod \big(image(}\mathcal{E}_{k-1},\nabla_{k-1}\big)\longrightarrow(\mathcal{E}_k,\nabla_k)).
\]
In other words, the matrix representation of $(\nabla_m)_w$ is triangular with diagonal entries $0$, $1,\cdots,m$. This implies that
its eigenvalues are $0$, $1$, $\cdots$, $m$.
\end{proof}

\subsection{Eigenspaces via symmetry} \label{EigenSym}

In Theorem \ref{T:resolving_sing}, we have seen that if
$A_a$ has eigenvalues $0$ and $m\in\mathbb{N}$, then there is a flat connection $(\mathcal{E},\nabla)$ over $X=\mathbb{P}^1\backslash\{0,1,\infty\}$ whose connection matrix is of the prescribed form
\[
-\big[\dfrac{A_0}{x}+\dfrac{A_1}{x-1}+\dfrac{A_a}{x-a}\big]\, dx.
\]
The following theorem shows that the condition $m\in\mathbb{N}$ and the choice of $A_a$ can be modified.

\begin{theorem}[(Resolving singularity - version II)]
\label{T:resolving_sing_b}
Let $X=\mathbb{P}^1\backslash\{0,1,\infty\}$ and
$a\in X$, and
let $A_0,A_1,A_a$ be $2\times 2$ matrices with complex entries such that
$0$ and $m\in\mathbb{Z}$ are the eigenvalues of $A_a$.
Then there exist a sheaf of locally free $\mathcal{O}_X$-module
$\mathcal{E}$, and a flat connection
$\nabla:\mathcal{E}\longrightarrow \mathcal{E}\otimes\Omega_X$ such that the connection matrix of $\nabla$ relative to a frame over an open subset of $X$ \emph{not containing} $a$ is given by
	\begin{equation}\label{E:nabla_2}
		-\Big[\dfrac{A_0}{x}+\dfrac{A_1}{x-1}+\dfrac{A_{a}}{x-{a}}\Big]\, dx.
	\end{equation}
\end{theorem}

\medskip
\begin{proof} If the eigenvalue $m$ of $A_a$ is a positive integer, the result is already dealt with in Theorem \ref{T:resolving_sing}. Suppose now that $0$ and $m$, where $-m\in\mathbb{N}$, are the eigenvalues of $A_a$. Then the matrix $A_a-mI$ has eigenvalues $0$ and $-m\in \mathbb{N}$. 
Thus Theorem \ref{T:resolving_sing} asserts that there is a sheaf of locally free $\mathcal{O}_X-$module $\mathcal{E}$, and a flat connection $\overline{\nabla}: \mathcal{E}\to \mathcal{E}\otimes \Omega_X$ whose matrix relative to a certain basis is
	\[
		 -\Big[\dfrac{A_0}{x}+\dfrac{A_1}{x-1}+\dfrac{A_{a}-mI}{x-{a}}\Big]\, dx.
	\]
	Then
	\[
		\big(\mathcal{E}, \overline{\nabla} - \frac{mI}{x-a}\, dx\big)
	\]
	gives the desired flat bundle with matrix representation \eqref{E:nabla_2}. Alternatively, one can obtain the same matrix representation \eqref{E:nabla_2} by considering the induced connection in $\mathcal{E}\otimes\mathcal{K}_a^m$.
Finally note that when $m=0$, $A_a$ has a repeated eigenvalue $0$. Together with the assumption that $A_a$ is diagonalisable at the end of Definition \ref{D:residue}, one sees that $A_a= 0$, i.e., a zero matrix, and the existence of of the flat bundle $(\mathcal{E},\nabla)$ over $X$ follows trivially.

\end{proof}
\bigskip

We next show that the conclusion of the above theorem continues to hold if the apparent singularity is located at any one of the singularities $x=0,\, 1,\, \infty$ instead of at $x=a$ via the Kummer symmetry.

\bigskip

\begin{theorem}[(Resolving singularity - version III)]
\label{T:resolving_sing_c}
Let $X=\mathbb{P}^1\backslash\{0,1,a,\infty\}$ and
let $A_0,A_1,A_a$ be $2\times 2$ matrices with complex entries such that the difference of the eigenvalues of one of the matrices from $\{A_0,\, A_1,\, A_a,\, A_\infty\}$ is an integer $m$. Then there exist a sheaf of locally free $\mathcal{O}_X$-module
$\mathcal{E}$, and a flat connection
$\nabla:\mathcal{E}\longrightarrow \mathcal{E}\otimes\Omega_X$ with three singularities only such that the connection matrix of $\nabla$ relative to a frame is
	\begin{equation}\label{E:nabla_3}
		-\Big[\dfrac{A_0}{x}+\dfrac{A_1}{x-1}+\dfrac{A_{a}}{x-{a}}\Big]\, dx.
	\end{equation}
\end{theorem}
\medskip

\begin{proof} Suppose that $j\in\{0,1,a\}$ and the two eigenvalues of $A_j$ are $\sigma$ and $\sigma+m$. Then the residue of the induced connection in $\mathcal{E}\otimes \mathcal{K}_j^{-\sigma}$ at $j$ has eigenvalues $0,\, m$. So without loss of generality, we may assume that one of the matrices $A_j\ (j\in\{0,\, 1,\, a\})$ to be considered below has one eigenvalue $0$ and another one $m$. Let us now suppose that $A_0$ has eigenvalues $0,\, m$. Then Theorem \ref{T:resolving_sing_b} asserts that there is a flat connection $(\mathcal{E},\, \nabla)$, with singularities at $0$, $1$ and $a/(a-1)$, whose connection matrix relative to a certain basis is
	\[
		-\Big[\frac{A_a}{x}+\frac{A_1}{x-1}+\frac{A_0}{x-\frac{a}{a-1}}\Big]\, dx.
	\]
Let $f(x)=(1-a)x+a$ be the M\"obius transformation that maps $(0,\, 1,\, \infty)$ to $(a,\, 1,\, \infty)$. Then the push-forward bundle $f_\ast\mathcal{E}$ induces a connection whose connection matrix assumes the desired form \eqref{E:nabla_3}.\\
\end{proof}	
\bigskip

The following theorem summarizes our discussion.

\begin{theorem}\label{apparenthyperplane}
Let a classical Riemann scheme 
	$P\begin{Bmatrix}0&1&a&\infty\\0&0&0&\alpha\\ 1-\gamma&1-\delta&1-\epsilon&\beta\end{Bmatrix}$
be given. If either one of the following condition holds:
		\begin{enumerate}
\item $\gamma\in\mathbb{Z}$; or
\item $\delta\in\mathbb{Z}$; or
\item $\epsilon\in\mathbb{Z}$; or
\item $\alpha-\beta\in\mathbb{Z}$,
		\end{enumerate}
then there exists a flat bundle over $\mathbb{P}^1$ with three points deleted whose Riemann scheme relative to an appropriately chosen frame is the
prescribed one.
\end{theorem}
\medskip

Thus, the necessary condition $\epsilon\in\N$ for the termination of the series (\ref{E:2F1h}) generates
other similar conditions via the Kummer symmetry, and we obtain
\bigskip

\begin{theorem}\label{hyperplanesH}
A necessary condition for the Heun equation \eqref{E:heun} to have an apparent singularity is when one of $\displaystyle\gamma,\delta,\epsilon,\alpha-\beta\in \mathbb{Z}\setminus\{0\}$.
\end{theorem}
\medskip

One can easily re-interpret the above eigenvalue problem for the following Heun connection.
\bigskip

\begin{corollary} \label{T:bad_eigenvalues}
A necessary condition for the existence of $\lambda$ such that the equation
	\begin{equation}\label{E:hyper4}
	x(x-1)\Big(\frac{dY}{dx}-\big[\dfrac{A_0}{x}+\dfrac{A_1}{x-1}+\dfrac{A_a}{x-a}\big]\Big)s=\lambda \, s,
\end{equation}
has a non-trivial solution $s$ consisting of a
finite $\sideset{_2}{_1}{\operatorname{F}}$ sum with a possible factor
$x^{b_1}(x-1)^{b_2}(x-a)^{b_3}$  is that
the differences of eigenvalues of $A_s$  is an integer for some $s=0,1,a$ or $\infty$.
\end{corollary}

\section{Type II degeneration: Monodromy reduction}\label{S:heunpolyn}

Recall the well-known theorem:

\begin{theorem}
If the Heun equation \eqref{E:heun} has a polynomial solution, then either $\alpha$ or $\beta$ is a non-positive integer.
\end{theorem}

For more detail, see for example, Ronveaux \cite[\S3.6]{Ronveaux}. This polynomial solution of (\ref{E:heun}) (i.e., (\ref{E:heun}) is reducible) is invariant under the monodromy of (\ref{E:heun}).

Now the focus of the upcoming study is the
Heun type connection with matrix relative to a frame being $-\Big[\dfrac{A_0}{x}+\dfrac{A_1}{x-1}+\dfrac{A_a}{x-a}\Big]\, dx$ and its
monodromy reduction at $\infty$  since $0$ is already an exponent at each of the singular points $0,\, 1,\, a$. The consideration of monodromy reduction at the points $0,\,1,\, a$ can  be transposed to $\infty$ later. 
We introduce the following definition
\begin{definition}\label{D:LR}
\begin{equation*}\tag{$\mathbf{LR}(m)$}
\Big\{\begin{array}{l}
\mbox{local condition for reducibility of monodromy at $\infty$:}\\
-m\in-\mathbb{N}\mbox{ is an eigenvalue of }A_{\infty}
\end{array}
\Big\}
\end{equation*}
\end{definition}

\begin{theorem} \label{reducible}
If $\nabla$ a connection of Heun-type \eqref{E:heun-scheme}
satisfying both $\mathbf{(WGRM)}$ and $(\mathbf{LR}(m))$ for some $m\in\mathbb{N}$, and if $\mathcal{L}$ is the  sheaf of horizontal sections of $\nabla$, then $\mathcal{L}$ is reducible.
\end{theorem}

\begin{proof}
Let $\nabla$ be a connection with 
matrix relative to a frame being $-\Big[\dfrac{A_0}{x}+\dfrac{A_1}{x-1}+\dfrac{A_a}{x-a}\Big]\, dx$
such that the  eigenvalues of $A_0,\, A_1,\, A_a$ are $\{0,\, 1-\gamma\}, \{0,\, 1-\delta\}, \{0,\, 1-\epsilon\}$ respectively. It follows from the assumption $(\mathbf{LR}(m))$ that $A_0+A_1+A_a$ has an eigenvalue $m\in\N$.
Now, $A_a$ has eigenvalues $0$ and $1-\epsilon$. By ($\mathbf{WGRM}$), $A_a$ and $-A_\infty=A_0+A_1+A_a$ are simultaneously diagonalisable. Hence $A_0+A_1$ has a positive integral eigenvalue $m$ (the case in which $1-\epsilon+m$ is an eigenvalue of $A_0+A_1$ will be apparent in Lemma \ref{heunreducible} below). Since  each of the three matrices $A_0,\, A_1,\, A_a$ has $0$ as an eigenvalue, Lemma \ref{hypergeometricreducible} implies that the $\exp(2\pi iA_0)$, $\exp(2\pi iA_1)$ and $\exp(2\pi i(A_0+A_1))$ have a common eigenvector $v$, say, which is also a common eigenvector of $A_0$, $A_1$ and $A_0+A_1$. 

But $A_a$ and $A_0+A_1$ are simultaneously diagonalisable,  hence they must have  common eigenvectors. But according to the last paragraph, $A_0+A_1$ shares an eigenvector with $A_0$ and $A_1$. We conclude that span$\{v\}$ is invariant under $A_0$, $A_1$ and $A_a$. Finally, the local monodromies of $\mathcal{L}$ are specified by $A_0$, $A_1$, $A_a$ and $-(A_0+A_1+A_a)$ which have a common eigenvector, so that $\mathcal{L}$ is reducible.
\end{proof}
\medskip

We immediately obtain the following corollary by following the argument in the proof of Theorem \ref{T:subbundle}. 
\medskip

\begin{corollary}\label{linesubbundle}
Let $(\mathcal{E},\nabla)$ be a connection of Heun type satisfying both $\mathbf{(WGRM)}$ and $(\mathbf{LR}(m))$ for some $m\in\mathbb{N}$. Then there exists a non-trivial morphism of flat bundles $(\mathcal{O}_X,d)\longrightarrow (\mathcal{E},\nabla)$.
\end{corollary}
\medskip

\begin{proof}
We essentially follow a similar procedure that led to Theorem \ref{T:subbundle} from Lemma \ref{hypergeometricreducible}.

We note that the residue matrices $A_0$, $A_1$, $A_a$ and $A_\infty$ that give rise to the monodromy matrices $P=e^{2\pi iA_0}$, $Q=e^{2\pi iA_1}$, $R=e^{2\pi iA_a}$, $PQR=e^{2\pi iA_\infty}$. 
\[
\begin{array}{ll}
P&\mbox{ has eigenvalues }1, e^{2\pi i(1-\gamma)}\\
Q&\mbox{ has eigenvalues }1, e^{2\pi i(1-\delta)}\\
R&\mbox{ has eigenvalues }1, e^{2\pi i(1-\epsilon)}\\
PQR&\mbox{ has eigenvalues }e^{2\pi i\alpha}, e^{2\pi i \beta}.
\end{array}
\]
Note that $-\alpha\in\mathbb{N}$.
According to the previous theorem, there is a common eigenvector $v$ for the above monodromy matrices (which represents a solution $f$ of the Heun equation defined on a small open set). Then after some routine consideration together with the Fuchsian condition, only two out of the total sixteen possibilities remain, which are, either
\[
\begin{array}{llll}
Pv=v&\mbox{ and }Qv=v&\mbox{ and } Rv=v&\mbox{ and } PQRv=v\\
\mbox{or}&&&\\
Pv=e^{2\pi i(1-\gamma)}v&\mbox{ and }Qv=e^{2\pi i(1-\delta)}v&\mbox{ and }Rv=e^{2\pi i(1-\epsilon)}v&\mbox{ and }PQRv=e^{-2\pi i\beta}v.\\
\end{array}
\]

The remaining steps of the proof go along a similar line of argument as those for the proof of the Theorem \ref{jacobi} with the new expression $x^{\gamma-1}(x-1)^{\delta-1}(x-a)^{\epsilon-1}f$ instead of $x^{c-1}(x-1)^{a+b-c}f$. We omit the details and conclude that the latter is eliminated. This completes the proof as the derivation of Theorem \ref{T:subbundle} from Theorem \ref{jacobi}.
\end{proof}

\subsection{Consequences of Kummer symmetry}

We may modify the requirements (\textbf{WGRM}) and ($\mathbf{LR}(m)$) assumed in  Theorem \ref{reducible} via symmetry as shown in the following corollary to the Theorem \ref{reducible}.
\medskip

\begin{corollary}\label{heunreducible}
For each $j\in\{0,1,a\}$, let $A_j$ be a $2\times 2$ matrix with $0$ as an eigenvalue. If
the following conditions hold:
\begin{itemize}
\item[\textbf{(L)}:] $m$ is an eigenvalue of $A_{\infty}=-A_0-A_1-A_a$ for some $m\in\mathbb{N}$; and
\item[\textbf{(G)}:]  there exist two matrices amongst $A_0$, $A_1$, $A_a$ and $A_{\infty}$ are simultaneously diagonalisable,
\end{itemize}
then the local system defined by the sheaf of horizontal sections of the connection with connection matrix
		\begin{equation}\label{E:nabla_4}
		-\Big[\dfrac{A_0}{x}+\dfrac{A_1}{x-1}+\dfrac{A_{a}}{x-{a}}\Big]\, dx
	\end{equation}
is reducible.
\end{corollary}
\medskip

\begin{proof}
We shall focus on the case in which $A_0$ and $A_1$ are simultaneously diagonalisable in the modified (\textbf{G}). The other
cases are handled in completely the same way.
Let $\mathcal{L}$ be the sheaf of horizontal sections of a connection with connection matrix \eqref{E:nabla_4}.
Let $f$ be a M\"obius transformation satisfying $f(0)=\infty$,
$f(a)=1$ and $f(\infty)=0$. Note that $f(x)=a/x$. In particular $f(1)=a$. Then $\mathcal{K}_0^m\otimes f_*\mathcal{L}$ is
the sheaf of horizontal sections of a connection with connection matrix
	\[
		-\Big[-\dfrac{A_0+A_1+A_a+mI}{x}+\dfrac{A_a}{x-1}+\dfrac{A_{1}}{x-{a}}\Big]\, dx.
	\]
Observe that $-(A_0+A_1+A_a+mI)+A_a+A_1=-A_0-mI$ and $A_1$ are simultaneously diagonalisable, and $A_0+mI$ has a positive integral eigenvalue $m$, so the local system $\mathcal{K}_0^m\otimes f_*\mathcal{L}$ is reducible by Theorem \ref{reducible}.
\end{proof}
\bigskip

Following the above argument, we now explore the full force of the Kummer symmetry to obtain a complete criterion for the reducibility of the monodromy of the Heun connections \eqref{E:Heun_connection}.
\medskip

\begin{theorem}\label{T:reduction}
For each $j\in\{0,1,a\}$, given a $2\times 2$ matrix $A_j$, let  $\lambda_{j1}$ and $\lambda_{j2}$ be the eigenvalues of $A_j$.
We also let $\lambda_{\infty 1}$, $\lambda_{\infty 2}$ be the eigenvalues of $A_{\infty}=-A_0-A_1-A_a$. If
	\begin{itemize}
		\item[\textbf{(L)}] for each $j\in\{0,1,a,\infty\}$, there exists $\lambda_j\in\{\lambda_{j1},\lambda_{j2}\}$ such that $\sum_j\lambda_j\in\mathbb{Z}$; and	
		\item[\textbf{(G)}] two matrices amongst $A_0$, $A_1$, $A_a$ and $A_{\infty}$ are simultaneously diagonalisable,
	\end{itemize}
then the sheaf of horizontal sections of a connection on $X=\mathbb{P}^1\backslash\{0,1,a,\infty\}$ with matrix relative to a frame being
	\[
		-\Big[\dfrac{A_0}{x}+\dfrac{A_1}{x-1}+\dfrac{A_{a}}{x-{a}}\Big]\, dx
	\]
is reducible.
\end{theorem}
\bigskip

\begin{proof}
Let $\mathcal{L}$ be the sheaf of horizontal sections of a connection with connection matrix
\[
		-\Big[\dfrac{A_0}{x}+\dfrac{A_1}{x-1}+\dfrac{A_{a}}{x-{a}}\Big]\, dx.
	\] Then
$\mathcal{L}\otimes\mathcal{K}_0^{-\lambda_0}\otimes\mathcal{K}_1^{-\lambda_1}\otimes\mathcal{K}_a^{-\lambda_a}$ is the sheaf of horizontal sections of a connection with connection matrix
\[
		-\Big[\dfrac{A_0-\lambda_0I}{x}+\dfrac{A_1-\lambda_1I}{x-1}+\dfrac{A_{a}-\lambda_aI}{x-{a}}\Big]\, dx.
	\]
Now for each $j$, $0$ is an eigenvalue of $A_j-\lambda_jI$, and $\sum_j(A_j-\lambda_jI)$ has an eigenvalue
\[
-\lambda_{\infty}-\sum_{j=0}^a\lambda_j\in\mathbb{Z}.
\]
Together with the hypothesis that two matrices among $A_0$, $A_1$, $A_a$ and $A_{\infty}$ are
simultaneously diagonalisable, the result follows from an application of Corollary \ref{heunreducible}.
\end{proof}
\bigskip

\begin{theorem}
For each $j\in\{0,1,a\}$, given $2\times 2$ matrices $A_j$  together with $A_\infty=-A_0-A_1-A_a$, suppose that they satisfy the conditions \textbf{(L)} and \textbf{(G)} in Theorem \ref{T:reduction}. If $(\mathcal{E},\nabla)$ is a flat bundle over $\mathbb{P}^1\backslash\{0,1,a,\infty\}$ with matrix relative to a frame being
    \[
        -\Big[\dfrac{A_0}{x}+\dfrac{A_1}{x-1}+\dfrac{A_{a}}{x-{a}}\Big]\, dx,
    \]
then there exist a rank-one flat bundle $(\mathcal{F},\nabla')$ and a non-trivial morphism $(\mathcal{F},\nabla')\to(\mathcal{E},\nabla)$.
\end{theorem}
\medskip

\begin{proof}
The sheaf of horizontal sections of $(\mathcal{E},\nabla)$ is reducible by Theorem \ref{T:reduction}. So one obtains the desired morphism by the procedure shown in Theorem \ref{T:subbundle}.
\end{proof}
\smallskip

A classical analogue is
\smallskip

\begin{theorem} \label{polysoln} Suppose that either one of $\alpha,\beta$ is in $\Z\setminus\{1\}$ or one of $\alpha-\gamma,\alpha-\delta,\alpha-\epsilon,\beta-\gamma,\beta-\delta,\beta-\epsilon$ is in  $\Z\setminus\{0\}$.
There exists $q$ such that the Heun equation (\ref{E:heun}) has a polynomial type solution, i.e., $x^{\tau_0}(x-1)^{\tau_1}(x-a)^{\tau_a}p(x)$
for some $\tau_0\in\{0,1-\gamma\}$, $\tau_1\in\{0,1-\delta\}$, $\tau_a\in\{0,1-\epsilon\}$ and
some polynomial $p(x)$.
\end{theorem}
\medskip

\begin{remark}
The first part of Theorem \ref{polysoln} can also be obtained from differential Galois theory \cite[p. 241]{DL}. However, our approach of the whole paper is from monodromy consideration of Heun equations directly instead of its simplification to differential Galois groups.   
\end{remark}

\section{Coincidence between Type I and Type II degenerations} \label{S:TwoType}

\subsection{Takemura's eigenvalues inclusion theorem} \label{S:Eig_In}

It follows from Theorem \ref{term1} that if $\epsilon=-m$, then there exist $m+1$ eigenvalues such that
the series (\ref{E:2F1h}) terminates. On the other hand, if $\alpha=-n$, then
there exist $n+1$ eigenvalues such that the series (\ref{E:2F1h}) becomes the Heun polynomial of degree $n$.
The following theorem of Takemura states that these two sets of the
eigenvalues have an inclusion relation.

\begin{theorem}[({\cite[Theorem 5.3]{Take4}})]\label{EigenInc}
Assume that $\epsilon$ and $\alpha$ are non-positive integers, but $\beta$ is not.
\begin{enumerate}
\item If $-\epsilon\geq-\alpha$ and the Heun equation (\ref{E:heun}) has a polynomial solution, then the singularity $x=a$ is apparent.
\item If $-\alpha\geq-\epsilon$ and the singularity $x=a$ of the Heun equation (\ref{E:heun}) is apparent, then the equation has a polynomial solution.
\end{enumerate}
\end{theorem}
\medskip

\subsection{Geometric interpretation of Takemura's proof}
While Takemura's method of proof is purely analytic, we shall establish a geometric argument that is based on the classification of monodromy reduction established in this paper that naturally leads to his result.
\medskip

\begin{theorem}\label{T:geo_takemura}
Let $a\in\mathbb{C}\backslash\{0,1\}$ and $X=\mathbb{C}\backslash\{0,1\}$. If $A_0$, $A_1$, $A_a$ are $2\times 2$ matrices with complex entries such that each matrix has $0$ as an eigenvalue and satisfy $(\mathbf{WGRM})$, $(\mathbf{WAS}(m))$ and $(\mathbf{LR}(n))$ for some $m$, $n\in\mathbb{N}$. Then there exist a rank-one flat bundle $(\mathcal{F}^\prime,\nabla')$ and a non-trivial morphism of flat bundles
\[
(\mathcal{F}^\prime,\nabla^\prime)\longrightarrow\Big(\mathcal{O}_{X\backslash\{a\}}\oplus\mathcal{O}_{X\backslash\{a\}},\ d-\big[\dfrac{A_0}{x}+\dfrac{A_1}{x-1}+\dfrac{A_a}{x-a}\big]dx\Big).
\]
\end{theorem}
\medskip

\begin{proof} Since ($\mathbf{WAS}(m)$) is satisfied, by a variation of Theorem \ref{T:resolving_sing} presented in Example \ref{E:bad_infinty},  there exist a rank-two flat bundle $(\mathcal{E},\nabla)$ over $X$ with
matrix relative to a frame being $-\Big[\dfrac{A_0}{x}+\dfrac{A_1}{x-1}\Big]\, dx$
and an isomorphism of flat bundles over $X\backslash\{a\}$
\[
(\mathcal{E},\nabla)\longrightarrow\Big(\mathcal{O}_{X\backslash\{a\}}\oplus\mathcal{O}_{X\backslash\{a\}},d-\big[\dfrac{A_0}{x}+\dfrac{A_1}{x-1}+\dfrac{A_a}{x-a}\big]\, dx\Big).
\]
Notice that the connection $\nabla$ above is of hypergeometric type, with $0$ as an eigenvalue of both of $A_0$ and $A_1$.  Thus the
 residue matrix of  $(\mathcal{E},\nabla)$ at $\infty$ is $-(A_0+A_1)=-(A_0+A_1+A_a)+A_a$. By the hypotheses ($\mathbf{LR}(n)$) and (\textbf{WGRM}), $A_0+A_1+A_a$ has an integral eigenvalue $n$ and $A_0+A_1+A_a$ and $A_a$ can be simultaneously diagonalised respectively. Thus, the hypothesis ($\mathbf{WAS}(m)$) implies that the sum $A_0+A_1$ has either $n-m$ or $n$ as an eigenvalue. In any case, one applies Theorem \ref{T:subbundle} to $(\mathcal{E},\nabla)$ to yield a rank-one flat bundle $(\mathcal{F}^\prime,\nabla')$ together with a nontrivial morphism
\[
(\mathcal{F}^\prime,\nabla')\longrightarrow(\mathcal{E},\nabla).
\]
The proof is finished by composing the two morphisms above.
\end{proof}
\medskip

\subsubsection*{Geometric proof of Takemura's theorem}
The assumption $-\alpha\in\mathbb{N}$ and the accessory parameter $q$ {being} chosen appropriately in  the Heun equation \eqref{E:heun_2} made in the Theorem \ref{EigenInc} (i) above correspond to  ($\mathbf{LR}(n)$) and  (\textbf{WGRM})  in our Theorem \ref{T:geo_takemura}. Of course, it is well-known that the Heun equation   \eqref{E:heun_2} admits a polynomial solution (see \cite{DL, Ronveaux}) under these assumptions.  The additional assumption that $-\epsilon\in\mathbb{N}$ made in Theorem \ref{EigenInc} (i) amongst to further assuming ($\mathbf{WAS}(m)$) in Theorem  \ref{T:geo_takemura} above.  The assumptions made in part (ii) of Theorem \ref{EigenInc} again amongst to be the same assumptions of Theorem  \ref{T:geo_takemura}, where the second assumption in Theorem \ref{EigenInc}(ii) corresponds to the assumptions ($\mathbf{WAS}(m)$) and (\textbf{WGRM}) in Theorem  \ref{T:geo_takemura}.

The conclusion of Theorem  \ref{T:geo_takemura} asserts that the corresponding Heun type connection is (singular) gauge equivalent to a hypergeometric type connection, which recovers Takemura's result.
\qed

\section{Matching to  Painlev\'e VI equation} \label{S:Pvi}
It is well-known that the Painlev\'e VI equation comes from an isomonodromic family of Heun equations
(see e.g., \cite{IKSY1991}, 
\cite{JM}).

In Sections \ref{S:eigenspaceI}, \ref{S:heunpolyn}, we have seen that two special types of Heun equations have special solutions.
The proofs relied heavily on the monodromies of the equations given. So it is interesting to
study special solutions of an \textit{isomonodromic family} of Heun equations
	\begin{equation}\label{E:PDE}
0=dY-\left[A_0\dfrac{dx}{x}+A_1\dfrac{dx}{x-1}+A_t\dfrac{d(x-t)}{x-t}\right]Y,
	\end{equation}
in which the eigenvalues of each of the three matrices are independent of $t$. When some of these eigenvalues are special, the solutions $Y(x,t)$ of the equation above are special functions in $x$. Our contribution in this section is to observe that the matrices $A_0,\, A_1,\, A_t$  are also special in $t$ when some of these eigenvalues are  chosen to satisfy certain arithmetic relations so that the monodromies degenerate.
Such a coincidence suggests that we study the $Y(x,\, t)$ to be special with respect to both  the variables $x,\, t$. This difficult work is reserved for future investigation.

Given a local function $Y(x)=\begin{bmatrix}y_1(x)\\y_2(x)\end{bmatrix}$
satisfying the following $2\times2$ Fuchsian system
\begin{equation}\label{E:theun}
	\nabla Y:=dY-\left[\dfrac{A_0}{x}+\dfrac{A_1}{x-1}+\dfrac{A_t}{x-t}\right]Y dx=0,
\end{equation}
let the matrices $A_0$,$A_1$,$A_t$ have the eigenvalues $(0,\theta_0)$ $(0,\theta_1)$
and $(0,1-\theta_t)$ respectively and $A_{\infty}:=-A_0-A_1-A_t=\begin{bmatrix}\kappa_1&0\\0&\kappa_2\end{bmatrix}$
is diagonal. Then $y_1$ satisfies the following equation (see Takemura \cite[\S2]{Take5})
\begin{eqnarray}\label{E:HeunIso}
        && \frac{d^2y}{dx^2}+
         \Big(\frac{1-\theta_0}{x}+\frac{1-\theta_1}{x-1}+\frac{1-\theta_t}{x-t}-\frac{1}{x-\lambda}\Big)\frac{dy}{dx}\nonumber\\
				&&\ +
          \left(\frac{\kappa_1(\kappa_2+1)}{x(x-1)}+\frac{\lambda(\lambda-1)\mu}{x(x-1)(x-\lambda)}-\frac{t(t-1)H}{x(x-1)(x-t)}\right)y=0,
\end{eqnarray}
where $\theta_\infty=\kappa_1-\kappa_2$, $\lambda$ is the zero of $(1,2)$-entry
of \[\dfrac{A_0}{x}+\dfrac{A_1}{x-1}+\dfrac{A_t}{x-t}:=
\begin{bmatrix}a_{11}(x)& a_{12}(x)\\a_{21}(x)& a_{22}(x)\end{bmatrix},\ \mu=a_{11}(\lambda)\] and
\[H=\frac{1}{t(t-1)}[\lambda(\lambda-1)(\lambda-t)\mu^2-{\theta_0(\lambda-1)(\lambda-t)+\theta_1\lambda(\lambda-t)
+(\theta_t-1)\lambda(\lambda-1)}\mu+\kappa_1(\kappa_2+1)(\lambda-t)].\]
The condition for isomonodromy deformation of the above equation is that $\lambda$ satisfies the sixth Painlev\'e equation ($P_{VI}$)
(also see \cite[\S3]{Mahoux}):

\[\frac{d\lambda}{dt}=\frac{\partial H}{\partial \mu}\quad\mbox{ and }\quad \frac{d\mu}{dt}=-\frac{\partial H}{\partial \lambda},\]
then it follows that
\beq
\frac{d^2\lambda}{dt^2}&=&\frac{1}{2}\bigg(\frac{1}{\lambda}+\frac{1}{\lambda-1}+\frac{1}{\lambda-t}\bigg)\left(\frac{d\lambda}{dt}\right)^2
-\bigg(\frac{1}{t}+\frac{1}{t-1}+\frac{1}{\lambda-t}\bigg)\frac{d\lambda}{dt}
+\frac{\lambda(\lambda-1)(\lambda-t)}{t^2(t-1)^2}\bigg(\frac{(1-\theta_\infty)^2}{2} \\&&\ +\frac{\theta_0^2t}{2\lambda^2}
 \  +\frac{\theta_1^2(t-1)}{2(\lambda-1)^2}+\frac{(1-\theta_t^2) t(t-1)}{2(\lambda-t)^2}\bigg).
\eeq

It is known that $\mathbf{P_{VI}}$ has special solutions expressed in terms of the hypergeometric functions
when $\theta_0$, $\theta_1$, $\theta_t$, $\theta_\infty$ satisfy the following conditions.
\begin{theorem}[{(\cite[Theorem 48.3]{GLS})}]
If either
\begin{equation}\label{eqn2}
\theta_0+\sigma_1\theta_1+\sigma_t\theta_t+\sigma_\infty\theta_\infty\in2\Z,
\end{equation}
for some $\sigma_1,\sigma_t,\sigma_\infty\in\{\pm1\}$ or
\begin{equation}\label{eqn3}
(\theta_0-n)(\theta_1-n)(\theta_t-n)(\theta_\infty-n)=0
\end{equation}
for some $n\in\mathbb{Z}$, then
$\mathbf{P_{VI}}$ has a one-parameter family of solutions expressed in terms of the hypergeometric functions.
\end{theorem}
On the other hand, under the above conditions (\ref{eqn2}) and (\ref{eqn3}),
we show that the equation \eqref{E:theun} has special solutions.

\begin{theorem}
Suppose that $\theta_0,\theta_1,\theta_t,\theta_\infty$ satisfy
\begin{itemize}
\item[(i)] the condition (\ref{eqn2}) for some $\sigma_1,\sigma_t,\sigma_\infty\in\{\pm1\}$. Then
the flat bundle endowed with the connection \eqref{E:theun} has a flat line subbundle.
\item[(ii)] the condition (\ref{eqn3}) for some $n\in\mathbb{Z}$. Then
the flat bundle endowed with the connection \eqref{E:theun} is isomorphic to one with one less singularity .
\end{itemize}
\end{theorem}
\begin{proof}[Sketch of Proof]
For (i), the equation (\ref{E:HeunIso}) has polynomial solutions only if either
\[\kappa_1=-\frac{\theta_0+\theta_1+\theta_t-\theta_\infty}{2}=0,-1,-2,\cdots \mbox{ or }\]
\[\kappa_2=-\frac{\theta_0+\theta_1+\theta_t+\theta_\infty}{2}=-1,-2,-3\cdots.\]
Moreover, other conditions of $\theta_0,\theta_1,\theta_t,\theta_\infty$ in (\ref{eqn2}) can
be obtained via the symmetry.

For (ii), the singularity $t$ in the equation (\ref{E:HeunIso}) is apparent only if $\theta_t=1,2,3,\cdots$. Moreover,
other conditions of $\theta_0,\theta_1,\theta_t,\theta_\infty$ in (\ref{eqn3}) can be obtained via the symmetry.
\end{proof}
\begin{example}
If $\kappa_1=-\frac{\theta_0+\theta_1+\theta_t-\theta_\infty}{2}=0$, take $\lambda=t$ and $\mu=0$,
then the equation (\ref{E:HeunIso})
has a constant solution.
\end{example}
\begin{example}
If $\theta_t=0$, take $\lambda=t$ and arbitrary $\mu$, then $A_t=0$ (see Takemura \cite{Take5} pp.19-20) and hence the point $t$ becomes an ordinary point so that the equation (\ref{E:HeunIso}) reduces to a hypergeometric equation.
\end{example}
\begin{remark} It is known that $\mathbf{P_{VI}}$ has a rational solution
if and only if $\theta_0,\theta_1,\theta_t,\theta_\infty$ satisfy both the condition
(\ref{eqn2})
for some $\sigma_1,\sigma_t,\sigma_\infty\in\{\pm1\}$ and
the condition (\ref{eqn3}) for some $n\in\mathbb{Z}$
(see \cite{VanAssche}).
\end{remark}

\section{Concluding remarks}\label{S:conclusion}
The Heun equation \eqref{E:heun} and its confluent forms are ubiquity in both mathematical physics and certain engineering disciplines since the early 19th century \cite{Ronveaux,WW}. It was also observed that special cases of  \eqref{E:heun} are closely related to several problems in number theory, see e.g., \cite{Chudnovsky1989}. 
In fact, there are already quite a number of published papers on various topics by researchers from very different interests and directions, e.g.,  \cite{Craster1996,Craster1997,CrasterHoang1998,SC, BM2015,CKLZ, Kimura, Maier}.
However,  the study of monodromy group of \eqref{E:heun} proves to be extremely difficult as seen from \cite{Erdelyi1,Erdelyi2}.  This is partly explained by the \textit{rigidity theory} proposed by Katz \cite{Katz1}.
Unlike the hypergeometric equation which is rigid, the Heun equation \eqref{E:heun} is \textit{not} rigid in general, meaning that it is essentially impossible to obtain closed form solutions to  \eqref{E:heun} in Euler-type integrals. Such difficulty of Heun equations is partly reflected in the recent book \cite{Ronveaux}. In view of the limitation of the use of classical mathematical language, the first main focus of this paper is to study  the Heun equation  as a special case of Fuchsian connections with four residue matrices, that is, to put the differential equation in the most natural geometric setting to where it belongs a prior. Besides, it turns out that for both theoretical interests and application purpose, it is important to consider the Heun equation  when one of its singularities becomes apparent, i.e., being resolved. But then the monodromy of \eqref{E:heun} is that of a hypergeometric connection. The equation becomes rigid, and closed form solutions in terms of hypergeometric functions become possible.

This paper studies the monodromy groups of special Fuchsian connections, namely the Heun connection \eqref{E:Heun_connection} with three categories of monodromy degenerations  (\textbf{WGRM}), (\textbf{WAS}$(m)$) and (\textbf{LR}$(n)$) from a sheave theoretic viewpoint.  In particular, we focus on the case when one of its singularities becomes apparent mentioned in the last paragraph. This allows us to give an interpretation of the classical infinite hypergeometric function expansions of  solutions to the Heun equation \eqref{E:heun} proposed by Erd\'elyi (1942, 1944) in terms of a sequence of appropriately defined injective bundle morphisms.  As a consequence, we have also derived new expansions which are not found in Erd\'elyi's papers. Oblezin applied Drinfeld's bundle modification technique to handle several well-known contiguous relations of hypergeometric functions.  This theory is now applied  to study properties of eigenvalues of \eqref{E:eigenvalue_prob_2}. The most startling finding here is that the eigenvalues are equally spaced. This is in stark contrast with the complicated behaviour of the   
 corresponding eigenvalues (accessory parameters) of \eqref{E:eigenvalue_prob_1} observed for the ``scalar" Heun equations \eqref{E:heun}. See, for example, the authors \cite{RA2010} have made a detailed numerical study on the dependence of the eigenvalues on the parameters of the Whittaker-Hill equation. The Whittaker-Hill equation is a trigonometric form of a confluent Heun equation. Indeed, very little is known about the analytic properties of these eigenvalues in general, even for the classical Mathieu equation and Lam\'e equation\footnote{The Lam\'e equation is an elliptic form and special case of the Heun equation where the local monodromy of three out of the four singularities are reduced, see, e.g.,  \cite{CCT}, while the Mathieu equation is a trigonometric form of a special confluent Heun equation.}. See  \cite{Borcea_Shapiro2008, Take6}  for recent qualitative descriptions of the eigenvalues for these periodic differential equations. Our finding shows that 
  either of the combinations between  \textbf{(WGRM})-(\textbf{WAS}$(m)$) or between  \textbf{(WGRM})-(\textbf{LR}$(n)$) leads to the monodromy degeneration of the Heun connection \eqref{E:Heun_connection}. The latter degeneration has essentially the same effect as the Heun connection being reducible which is well studied from the viewpoint of differential Galois theory  \cite{DL}. From our viewpoint, the former appears to play a more fundamental role in describing  the monodromy reduction of the Heun connection.  The study of hypergeometric equations with several additional apparent singularities by Shiga, Tsutsui and Wolfart \cite{STW} also falls within this category (see Example \ref{E:theorists}).
  We have shown that when both the two reduction modes happen simultaneously, then there exists a rank-one flat bundle and a non-trivial morphism that sends the rank-one flat bundle to the flat bundle of a Heun connection. An immediate consequence is  a geometric proof of Takemura's theorem which describes a certain inclusion of eigenvalues of two types of monodromy reductions, due to the common origin (\textbf{WGRM}). Our next unexpected observation is that the combined criteria of monodromy reductions   
  (\textbf{WAS}$(m)$) and (\textbf{LR}$(n)$) of the Heun connection that we have enumerated \textit{matches precisely} the criteria for the degeneration of Painlev\'e equation VI, that is, when the $\mathbf{P_{VI}}$ either admits a rational solution or in terms of hypergeometric functions.  It is well-known from the work of Okamoto that a Weyl group acts on the parameter space of  the degenerated $\mathbf{P_{VI}}$. 
  Such matching may reasonably be studied from the viewpoint of the PDE \eqref{E:PDE}. Finally, we would like to mention that the existence of  orthogonality between two eigen-solutions of terminated hypergeometric sums of the form \eqref{E:erdelyi}, i.e., from the degeneration of (\textbf{WAS}$(m)$),
 appears to be a non-trivial problem. The case for the \textit{joint orthogonality} for Heun polynomials, i.e., from the  combined degenerations of  (\textbf{(WGRM}) and (\textbf{LR}$(n)$), has been studied recently by Felder and Willwacher \cite{FW2015}.
   We hope to return to  these problems in the near future.

{\appendix

\section{Fuchsian relations}\label{A:fuchsian}
We point out that there is a difference between the Fuchsian relation of a Fuchsian connection and a Fuchsian differential equation \cite{beukers2007} that is derived from the Fuchsian connection.  Since we cannot find a reference for this fact, so a proof is provided here.

Let  a connection $\nabla$ relative the canonical basis have the matrix representation
	\begin{equation}\label{E:connection}
		-\sum_{j=1}^{n} \frac{A_j}{x-a_j}:=-A(x)
		=\left(
			\begin{matrix}
				A_{11} & A_{12}\\
				A_{21} & A_{22}
			\end{matrix}
		\right)							
	\end{equation}
	such that $a_{n+1}=\infty$ and that the residue matrices $A_j\ (j=1,\cdots, n+1)$ satisfy the \textit{Fuchsian relation}:
	\begin{equation}\label{E:fuchsian_relation_1}
		\sum_{j=1}^{n+1} \Tr(A_j)=0.
	\end{equation}
\bigskip

\begin{lemma}\label{L:fuchsian_relation} Let $\Tr(A_j)=\alpha_{j1}+\alpha_{j2}$ be the trace of the residue matrix $A_j\ (j=1,\cdots n+1)$. Let $y$ be the first component of a horizontal section of the connection defined above. Let the Riemann scheme of the differential equation 
	\begin{equation}
		\label{E:fuchsian_eqn}
			y^{\prime\prime}-\big(\Tr \big(A(x)\big)+\frac{A^\prime_{12}(x)}{A_{12}(x)}\big)y^\prime+\Big(\det A(x)-A_{11}(x)\log ^\prime\big(A_{12}(x)/A_{11}(x)\big)\Big)y=0
	\end{equation}
derived from the connection above and satisfied by $y$ be of the form
	\[
		P
		\left(
		\begin{matrix}
		a_1               & \cdots   &a_j             &\cdots      & a_n  & \infty &b_1 & \cdots &b_{n-1}           \\
		\beta_{11}    & \cdots    &\beta_{j1}   & \cdots    & \beta_{n1}  &\beta_{n+1,\, 1} 
		& 0 & \cdots & 0\\
		\beta_{22}    & \cdots    &\beta_{j2}   & \cdots    & \beta_{n2} &\beta_{n+1,\, 2}
		& 2 & \cdots & 2
		\end{matrix}
		;\ x
		\right),
	\]where $b_j\ (j=1,\cdots, n-1)$ are the apparent singularities. If $b_j=a_j\ (j=1,\cdots, n-1)$, then we have
	\begin{equation}\label{E:fuchsian_comparison}
		\begin{split}
			&\alpha_{j1}+\alpha_{j2} =\beta_{j1}+\beta_{j2}-1\quad (j=1,\cdots, n-1)\\
			 &\alpha_{n1}+\alpha_{n2} =\beta_{n1}+\beta_{n2},\\
			 &\alpha_{n+1,\, 1}+\alpha_{n+1,\, 2} =\beta_{n+1,\, 1}  +\beta_{n+1,\, 2}.
		\end{split}
	\end{equation}
	In particular,
		\[
			\sum_{j=1}^{n+1} (\beta_{j1}+\beta_{j2})=(n-1)+\sum_{j=1}^{n+1} \Tr(A_j).
		\]
\end{lemma}
\medskip
\begin{proof} It is sufficient to compute the coefficient of $y^\prime$. Let
	\[
		A_{12}(x)=:\mathrm{(const.)} \frac{\prod_{j=1}^{n-1}(x-b_j)}{\prod_{j=1}^n(x-a_j)}.
	\]Hence
	\begin{equation}
		-\Tr \big(A(x)\big)-\frac{A^\prime_{12}(x)}{A_{12}(x)}
		=\sum_{j=1}^n \frac{1-\alpha_{j1}-\alpha_{j2}}{x-a_j} -\sum_{j=1}^{n-1}\frac{1}{x-b_j}
	\end{equation}
where the $b_j\  (j=1,\cdots, n-1)$ are apparent singularities which are inherited from $A^\prime_{12}/A_{12}$ and they do not contribute to the monodromy of the equation  \eqref{E:fuchsian_eqn}. Let us now assume $b_j=a_j\ (j=1,\cdots, n)$. Hence
	\[
		\begin{split}
		 -\Tr \big(A(x)\big)-\frac{A^\prime_{12}(x)}{A_{12}(x)}
		&=\sum_{j=1}^{n-1} \frac{-\alpha_{j1}-\alpha_{j2}}{x-a_j}+\frac{1-\alpha_{n1}-\alpha_{n2}}{x-a_n} \\
		&=\sum_{j=1}^{n-1} \frac{1-\beta_{j1}-\beta_{j2}}{x-a_j}+\frac{1-\beta_{n1}-\beta_{n2}}{x-a_n},
		\end{split}
	\]
	where we have identified $\alpha_{j1}+\alpha_{j2}=\beta_{j1}+\beta_{j2}-1\ (j=1,\cdots, n-1)$, $\alpha_{n1}+\alpha_{n2}=\beta_{n1}+\beta_{n2}$ and $\alpha_{n+1,\, 1}+\alpha_{n+1,\, 2} =\beta_{n+1,\, 1}  +\beta_{n+1,\, 2}$. Thus the sum of the indicial roots of the Riemann scheme $P$ above yields
	\[
		\begin{split}
		\sum_{j=1}^{n+1} (\beta_{j1}+\beta_{j2})
		&=\sum_{j=1}^{n-1} (\beta_{j1}+\beta_{j2}) +(\beta_{n1}+\beta_{n2}) 
		+(\beta_{n+1,\, 1}+\beta_{n+1,\, 2})\\
		&= n-1+\sum_{j=1}^{n-1} (\alpha_{j1}+\alpha_{j2}) +(\alpha_{n1}+\alpha_{n2}) 
		+(\alpha_{n+1,\, 1}+\alpha_{n+1,\, 2})\\
		&=n-1+\sum_{j=1}^{n+1} \Tr(A_j).
		\end{split}
	\]
	\end{proof}
\medskip

\begin{remark} It is clear that one can also include $\infty$ as one of the apparent singular points in the above theorem. We show below an example in terms of a hypergeometric equation where the apparent singularity is placed at infinity.
\end{remark}
\medskip

\begin{example}
Consider $Y(x)=\begin{bmatrix}y_1(x)\\y_2(x)\end{bmatrix}$
satisfying the following $2\times2$ Fuchsian system
\begin{equation}\label{E:hyper1}
\frac{dY}{dx}-[\dfrac{A_0}{x}+\dfrac{A_1}{x-1}]Y=0,
\end{equation} corresponding to the connection of hypergeometric type ($A_0,A_1$), where
$A_0=\begin{bmatrix}u_0+\gamma&1\\-u_0(u_0+\gamma)&-u_0\end{bmatrix}$,
$A_1=\begin{bmatrix}u_1+\delta&-1\\u_1(u_1+\delta)&-u_1\end{bmatrix}$,
$u_0=\frac{\alpha(\alpha+\gamma)}{\alpha-\beta}$,
$u_1=\frac{\alpha(\alpha+\delta)}{\alpha-\beta}$ such that
\[A_\infty=-A_0-A_1
          =\begin{bmatrix}\beta&0\\0&\alpha\end{bmatrix}, \mbox{ where } \alpha+\beta=\gamma+\delta,\]
then $y_1(x)$ satisfies
\begin{equation}\label{E:hyper2}
         \frac{d^2y}{dx^2}+
         \Big(\frac{1-\gamma}{x}+\frac{1-\delta}{x-1}\Big)\frac{dy}{dx}+
         \frac{\beta(\alpha+1)}{x(x-1)}y=0.
\end{equation} and
$y_2(x)$ satisfies
\begin{equation}\label{E:hyper3}
         \frac{d^2y}{dx^2}+
         \Big(\frac{1-\gamma}{x}+\frac{1-\delta}{x-1}\Big)\frac{dy}{dx}+
         \frac{\alpha(\beta+1)}{x(x-1)}y=0.
\end{equation}
Two linearly independent local solutions are identified as
$\begin{bmatrix}y_1&y_2\end{bmatrix}^T$
and $\begin{bmatrix}\tilde{y_1}&\tilde{y_2}\end{bmatrix}^T$, where
\[y_1(x)=\sideset{_2}{_1}{\operatorname{F}}\left({\begin{matrix}
                 \alpha+1,\beta\\
                  \gamma\end{matrix}};x\right),
\	y_2(x)=C\sideset{_2}{_1}{\operatorname{F}}\left({\begin{matrix}
                 \alpha,\beta+1\\
                  \gamma\end{matrix}};x\right)\]
and
\[\tilde{y_1}(x)=x^{1-\gamma}\sideset{_2}{_1}{\operatorname{F}}\left({\begin{matrix}
                 \alpha-\gamma+2,\beta-\gamma+1\\
                  2-\gamma\end{matrix}};x\right),
\	\tilde{y_2}(x)=\hat{C}x^{1-\gamma}\sideset{_2}{_1}{\operatorname{F}}\left({\begin{matrix}
                 \alpha-\gamma+1,\beta-\gamma+2\\
                  2-\gamma\end{matrix}};x\right).\]
\end{example}

\section{Convergence of Erd\'elyi's expansions}\label{A:erdelyi}
The Heun function can be expanded as
	\begin{equation}\label{E:2F1h}
		\begin{split}
		Hl(a,\, q;\, \alpha,\, &\beta,\, \gamma,\, \delta;\, x)
		=\sum_{m=0}^\infty X_m\varphi_m^1(x)\\
		&=\sum_{m=0}^\infty X_m
		\frac{\Gamma(\alpha-\delta+m+1)\Gamma(\beta-\delta+m+1)}
{\Gamma(\alpha+\beta-\delta+2m+1)}
x^m
\sideset{_2}{_1}{\operatorname{F}}\left({\begin{matrix}
                 \alpha+m,\beta+m\\
                  \alpha+\beta-\delta+2m+1
                 \end{matrix}};x\right), 
                 \end{split}
	\end{equation}
whose
coefficients $X_m$ satisfy a three-term recursion relation
\begin{equation}\label{E:Three_term}
\begin{cases}
L_0X_0+M_0X_1=0\\
K_mX_{m-1}+L_mX_m+M_mX_{m+1}&=0,\ m=1,2,\cdots
\end{cases},
\end{equation}
where $K_m,L_m,M_m$ are given in \cite[(5.3)]{Erdelyi1}
	\[
		\begin{split}
			K_{m+1}& :=a\frac{(\alpha+m)(\beta+m)(\varepsilon+m)(\alpha+\beta-\delta+m)}{(\alpha+\beta-\delta+2m)(\alpha+\beta-\delta+2m+1)}\\
			L_m\ & :=am(\gamma+m-1)
			\Big\{\frac{(\alpha+m)(\alpha-\delta+m+1)+(\beta+m)(\beta-\delta+m+1)}
			{(\alpha+\beta-\delta+2m-1)(\alpha+\beta-\delta+2m+1)}\\
			&\quad - \frac{1}{\alpha+\beta-\delta+2m-1}\Big\}-
			m(\alpha+\beta-\delta+m)-\alpha\beta q\\
			&\quad +a\frac{\alpha\beta(\gamma+2m)-\varepsilon m (\delta-m-1)}
			{\alpha+\beta-\delta+2m+1)}\\
			M_{m-1}&:= \frac{am(\alpha-\delta+m)(\beta-\delta+m)(\gamma+m-1)}
			{(\alpha+\beta-\delta+2m-1)(\alpha+\beta-\delta+2m)}.
		\end{split}
	\]
			
\bigskip

The convergence of the above series is given by
\medskip

\begin{theorem}[\cite{Erdelyi1}]\label{T:convergent} \label{BiggerDomain}
Suppose that the series (\ref{E:2F1h}) is non-terminating, and the branch of the square root is chosen such that its real part is nonnegative.
Let $k=\bigg|\frac{1-\sqrt{1-a}}{1+\sqrt{1-a}}\bigg|\neq1$. Then it converges {uniformly} compacta on
		\begin{equation}\label{E:omega_0}
			\Omega_0=\Big\{x\in\mathbb{C}:\bigg|\frac{1-\sqrt{1-x}}{1+\sqrt{1-x}}\bigg|<\min(k,k^{-1})\Big\},
		\end{equation}
where $\Omega_0$ denotes a neighbourhood of $0$ but excluding $x=1$ (see the Remark below). Moreover, if the accessory parameter $q$ in (\ref{E:heun}) satisfies the infinite continued fraction
\begin{equation}\label{E:cf}
L_0/M_0-\frac{K_1/M_1}{L_1/M_1-}\frac{K_2/M_2}{L_2/M_2-}\frac{K_3/M_3}{L_3/M_3-}\cdots=0,
\end{equation}
which contains $q$ implicitly, then the series (\ref{E:2F1h}) converges in a larger region
	\begin{equation}\label{E:omega_1}
		\Omega_1=\Big\{x\in\mathbb{C}:\bigg|\frac{1-\sqrt{1-x}}{1+\sqrt{1-x}}\bigg|<\max(k,k^{-1})\Big\}
	\end{equation}
except possibly on the branch cut $[1, +\infty).$
\end{theorem}
\medskip

We include a brief proof since the argument may not be easily found in modern literature.

\bigskip
\begin{proof}
We apply Poincar\'e's Theorem and Perron's Theorem  to the three-term recurrence relation \eqref{E:Three_term}, see e.g.,  \cite[Theorems 1.1, 2.1-2.2]{Gautschi} to yield,
\[\lim_{m\to\infty}\left|\frac{X_{m+1}}{X_{m}}\right|=
\begin{cases}
\min(k,k^{-1}) &\mbox{if \eqref{E:cf} holds}\\
\max(k,k^{-1}) &\mbox{otherwise}
\end{cases}.\]
After applying Watson's asymptotic representation (see \cite[\S9]{Watson}), we derive
\begin{equation}\label{E:asym}
\frac{\varphi_{m+1}^1(x)}{\varphi_{m}^1(x)}\sim \frac{1-\sqrt{1-x}}{1+\sqrt{1-x}}\
\mbox{ as } m\to\infty
\end{equation}
(see \cite[(4.7)]{Erdelyi1}). Thus, the result follows from the ratio test applied to the cases $\Omega_0$ and $\Omega_1$ separately.
\end{proof}
\bigskip

\begin{remark}[(Description of $\Omega_0$ and $\Omega_1$)] \label{A:remark_1}
Note that for $m>0$, $\big|\frac{1-y}{1+y}\big|=m$ is equivalent to the circle equation 
\[|y-y_0|=r,
\mbox{ where } y_0=\frac{1+m^2}{1-m^2}\mbox{  and } r=\frac{2m}{|1-m^2|}.
\]
Let $m_0:=\min(k,k^{-1})<1$ and $m_1:=\max(k,k^{-1})>1$. Then
\[\Big\{y\in\mathbb{C}:\bigg|\frac{1-y}{1+y}\bigg|<m_0\Big\}\] is the open disk $D_0$ centred at $y_0=\frac{1+m_0^2}{1-m_0^2}>0$ with radius $r=\frac{2m_0}{1-m_0^2}$ containing $y=1$. In particular, $D_0$ is contained in the half-plane $\{\mbox{Re }y>0\}$. Since $\mbox{Re}\sqrt{1-x}$ is always taken to be nonnegative, $\Omega_0$ is a neighborhood of $0$, not containing $x=1,\infty$.
On the other hand, \[\Big\{y\in\mathbb{C}:\bigg|\frac{1-y}{1+y}\bigg|<m_1\Big\}\] is the complement of the closed disk $D_1$ centred at $y_0=-\frac{m_1^2+1}{m_1^2-1}<0$ with radius $r=\frac{2m_1}{m_1^2-1}$. In particular, the complement contains the half-plane $\{\mbox{Re }y\geq0\}$. Since $\mbox{Re}\sqrt{1-x}$ is always taken to be nonnegative, $\Omega_1=\mathbb{C}\backslash [1,\, \infty)$.
\end{remark}

\begin{remark}[(A second linearly independent series solution)]\label{A:remark_2}
\label{other_soln}
Erd\'elyi also considered the series solution $\displaystyle\sum_{m=0}^\infty X_m\varphi_m$ other than \eqref{E:2F1h} by replacing $\varphi_m^1$ with another linearly independent solution $\varphi_m$, which can be any linear combination of $\varphi_m^2,\cdots,\varphi_m^6$ defined in \cite[Eqn(4.2)]{Erdelyi1}. In this case, we have the following asymptotic representation (see \cite[Eqn(4.8)]{Erdelyi1}) instead of \eqref{E:asym}
	\begin{equation*} 
		\frac{\varphi_{m+1}(x)}{\varphi_{m}(x)}\sim \frac{1+\sqrt{1-x}}{1-\sqrt{1-x}}\
\mbox{ as } m\to\infty.
	\end{equation*} 
In order to study the domain of convergence, we consider
$$\Omega_0^-=\Big\{x\in\mathbb{C}:\bigg|\frac{1+\sqrt{1-x}}{1-\sqrt{1-x}}\bigg|<m_0\Big\}
\mbox{ and }
\Omega_1^-=\Big\{x\in\mathbb{C}:\bigg|\frac{1+\sqrt{1-x}}{1-\sqrt{1-x}}\bigg|<m_1\Big\}.
$$
Note that
\[\Big\{y\in\mathbb{C}:\bigg|\frac{1+y}{1-y}\bigg|<m_0\Big\}\]
is the open disk $D_0^-$ centred at $y_0=-\frac{1+m_0^2}{1-m_0^2}<0$ with radius $r=\frac{2m_0}{1-m_0^2}$ containing $y=-1$, and
\[\Big\{y\in\mathbb{C}:\bigg|\frac{1+y}{1-y}\bigg|<m_1\Big\}\]
is the complement of the closed disk $D_1^-$ centred at $y_0=\frac{m_1^2+1}{m_1^2-1}>0$ with radius $r=\frac{2m_1}{m_1^2-1}$.
In particular, $D_0^-$ is contained in the half-plane $\{\mbox{Re } y<0\}$, and $D_1^-\subseteq\{\mbox{Re } y>0\}$ contains $y=1$, but not $y=0,\, \infty$. Since $\mbox{Re}\sqrt{1-x}$ is always taken to be nonnegative, $\Omega_0^-=\emptyset$ and $\Omega_1^-$ is the domain containing $x=1,\infty$, but not $x=0$. As the coefficients $X_m$ satisfy the same three-term recurrence relation \eqref{E:Three_term}, by the similar argument in the proof of Theorem \ref{T:convergent}, the series 
\[
\begin{cases}
\mbox{converges on $\Omega_1^-$} &\mbox{if \eqref{E:cf} holds}\\
\mbox{diverges} &\mbox{otherwise}.
\end{cases}
\]
We conclude that the two series $\sum_{m=0}^\infty X_m\varphi_m^1(x)$ and $\sum_{m=0}^\infty X_m\varphi_m$ both converge in $\Omega_1^-$ when \eqref{E:cf} holds.
\end{remark}
}

\section*{Acknowledgment} The authors would like to express their gratitude towards the referee for his/her critical comments that, amongst others, cleared up some ambiguities  of our original manuscript.

\section*{Funding support}
	\begin{itemize}
		\item  The first and the third authors were partially supported by the Research Grants Council of Hong Kong (No. 16300814).
	\end{itemize}
	
\bibliographystyle{abbrv}

\bibliography{referencesHeun}

\end{document}